\documentclass[final]{siamart190516}
\usepackage{lipsum}
\usepackage{amsfonts}
\usepackage{graphicx}
\usepackage{epstopdf}
\usepackage{algorithmic}
\usepackage{amsmath}
\usepackage{amssymb}
\usepackage{subfigure}
\usepackage{booktabs}
\usepackage{multirow}
\usepackage{amsopn}

\ifpdf
\DeclareGraphicsExtensions{.eps,.pdf,.png,.jpg}
\else
\DeclareGraphicsExtensions{.eps}
\fi

\ifpdf
\hypersetup{
	pdftitle={},
	pdfauthor={}
}
\fi

\newsiamremark{remark}{Remark}
\newsiamremark{hypothesis}{Hypothesis}
\crefname{hypothesis}{Hypothesis}{Hypotheses}
\newsiamthm{claim}{Claim}

\begin{document}
\headers{MMC homogenization method for random material}{Zihao Yang, Jizu Huang, Xiaobing Feng, Xiaofei Guan}

\title{An efficient multi-modes Monte Carlo homogenization method for random materials
}


\author{Zihao Yang\thanks{School of Mathematics and Statistics, Northwestern Polytechnical University, Xi'an 710072, China (\email{yangzihao@nwpu.edu.cn}). The work of this author was partially supported by
			the National Natural Science Foundation of China grant 11871069 and by the National Key R\&D Program of China with the grant 2020YFA0713603.}
	\and Jizu Huang\thanks{Corresponding author. LSEC, Academy of Mathematics and Systems Science, Chinese Academy of Sciences, Beijing 100190, China, and School of Mathematical Sciences, University of Chinese Academy of Sciences, Beijing 100049, China
		(\email{huangjz@lsec.cc.ac.cn}).}
	\and Xiaobing Feng\thanks{Department of Mathematics, The University of Tennessee, TN 37996, USA   (\email{xfeng@math.utk.edu}).}
	\and Xiaofei Guan\thanks{School of Mathematical Sciences, Tongji University, Shanghai 200092, China   (\email{guanxf@tongji.edu.cn}).}
}

\maketitle

\begin{abstract}
In this paper, we propose and analyze a new stochastic homogenization method for diffusion equations
with random and fast oscillatory coefficients. In the proposed method, the homogenized solutions are
sought through a two-stage procedure. In the first stage, the original oscillatory diffusion equation is approximated, for each fixed random sample $\omega$, by a spatially homogenized diffusion equation with
piecewise constant coefficients, {\color{black}resulting in} a random diffusion equation.
 In the second stage,
the {\color{black}resulting} random diffusion equation is approximated and computed by using an efficient
multi-modes Monte Carlo method  which only requires to solve a  diffusion equation
with a constant diffusion coefficient and a random right-hand side.
The main advantage of the proposed method is that it separates the
computational difficulty caused by the spatial fast oscillation of the solution and that caused by the randomness of the solution, so they can be overcome separately using different strategies.
The convergence of the solution of the spatially homogenized equation (from the first stage) to the solution of the original random diffusion equation is established and the optimal rate of convergence is also obtained for the proposed multi-modes Monte Carlo method.
Numerical experiments on some benchmark test problems for random composite materials are also
presented to gauge the efficiency and accuracy of the proposed two-stage stochastic homogenization method.
\end{abstract}



	\section{Introduction}\label{sec-intro}
	This paper is concerned with numerical solutions of the following diffusion equation
	with random coefficients and data encountered in materials science:
	\begin{subequations}\label{eq:problem}
		\begin{alignat}{2}
		-\hbox{div}\Bigl(A \biggl(\frac{x}{\varepsilon},\omega\biggr)\nabla u^\varepsilon(x,\omega) \Bigr)&=f(x,\omega)  &&\qquad\mbox{in } {D}\times \Omega, \\
		u^\varepsilon(x,\omega) &= 0 &&\qquad\mbox{on } \partial {D}\times \Omega.
		\end{alignat}
	\end{subequations}
	Here ${D} \subset \mathbb{R}^d (d=1,2,3)$ is a bounded domain and $\omega$ denotes a sample point
	which belongs to a probability (sample) space $(\Omega, \mathcal{F}, \mathbb{P})$.
	The coefficient matrix $A(\frac x \varepsilon,\omega)=(a_{ij}(\frac x \varepsilon,\omega))_{1\leq i,j \leq d}$ and the right hand side term $f(x,\omega)$ are random fields with continuous and bounded covariance
	functions.
   {The parameter $\varepsilon$ represents the size of microstructure for the composite materials, which is usually very small, that is,  $0< \varepsilon \ll 1$.}

	The random diffusion equation \eqref{eq:problem} has many applications in mechanics, hydrology and thermics (see \cite{graham2011quasi, 2014Efficient, suribhatla2011effective,  to2020fft}).
	A direct accurate numerical solution of \eqref{eq:problem} is difficult to obtain because it
	requires a very fine mesh and large-scale sampling of $\omega$, and thus a prohibitive amount
	of computation time. In the case of absence of the randomness (i.e., the dependence on $\omega$ in \eqref{eq:problem} is dropped ), the homogenization method has been successfully developed for solving the diffusion equation with periodic deterministic coefficients
	(cf. \cite{bensoussan2011asymptotic, cioranescu1999introduction, tartar2009general}),
	in which the homogenized coefficients are obtained by solving a cell problem defined in the unit cell.
	For the diffusion equation \eqref{eq:problem} with  random  coefficients,
	a stochastic homogenization theory has also been well developed, see
	\cite{blanc2007stochastic,
		bourgeat2004approximations, duerinckx2016structure, cui2005multi,
		gloria2012optimal, li2005multi, kozlov1979averaging, papanicolaou1979boundary, sab1992homogenization, vel2010multiscale, wu2016Efficient}.
	Similar to the deterministic case, the homogenized diffusion equation is constructed by
	solving a certain cell problem. However, a fundamental difference is that the stochastic
	cell problem is a random second order elliptic problem, which is posed in the
	whole space $\mathbb{R}^d$ (see \eqref{eq:cell-problem}).  Solving such an infinite domain problem
	numerically is not only challenging but also very expensive, and the homogenization methods quoted above did not give any practical recipe for numerically approximating the cell problem and the homogenized equation.

	To circumvent the above difficulties, some localized approximations of the effective (or homogenized) coefficients by using ``periodization" and  ``cut-off" procedures were introduced in
	\cite{anantharaman2012introduction, bourgeat2004approximations,
		pozhidaev1990error, yurinskii1986averaging}.
	A big benefit of the localized approximations is that the resulting cell problem is now posed
	on a bounded domain and it was proved that the approximated coefficients
	converge to the effective (or homogenized) coefficients as the size of the bounded domain goes to infinity.
	We also note that the localized approximation methods, such as the representative volume element (RVE) method, have also been used to compute the effective parameters of highly heterogeneous materials and to calculate
	the effective coefficients related to random composite materials, by utilizing a possibly large number of realizations \cite{cui2005multi, gitman2007representative,li2005multi,  xu2009stochastic}.
	After having constructed the approximated effective (or homogenized) coefficients, the main task then reduces  to solve the approximated (random) diffusion equation with the constructed coefficients.
    {As a direct application of Monte Carlo or stochastic Galerkin method, since solving this random diffusion equation is computationally expensive, other more efficient methods have been
	developed for the job.} In \cite{anantharaman2011numerical, anantharaman2012elements},
	a perturbative model for weakly random materials was proposed and the first-order and second-order
	asymptotic expansions were established by means of an ergodic approximation based on the weak randomness
	assumption. {\color{black}In \cite{feng2016multimodes},} a multi-modes Monte Carlo (MMC) finite element method was
	proposed to solve random {\color{black}partial differential equations (PDEs)} under the assumption that the medias (or coefficients) are weakly random
	in the sense that they can be expressed as small random perturbations of a deterministic background.
	However, to the best of our knowledge, there is still no efficient method for solving the homogenized
	equation for problem \eqref{eq:problem} in general setting.

	The goal of this paper is to develop and analyze an efficient and practical two-stage
	homogenization method for problem \eqref{eq:problem}.  In the first stage, we construct
	a piecewise homogeneous (i.e., piecewise constant) material as an approximation to the
	original composite material for each fixed sample $\omega$ by solving several cell functions
	on bounded domains. Under the stationary (process) assumption, we are able to prove that the coefficient matrix
	of the piecewise homogeneous material can be rewritten as a small random perturbation of a
	deterministic matrix, which then sets the stage for us to adapt the MMC framework.
	In the second stage, we utilize the MMC finite element method to solve
	random diffusion problem with the piecewise coefficients obtained from the first stage and
	provide a complete convergence analysis for the MMC method.
	Since the first stage of the proposed method is similar to the RVE method, the work of this
	paper can be regarded as a mathematical interpretation and justification for the RVE method
	for problem \eqref{eq:problem}
	and introduces a  numerical framework for an efficient implementation of the RVE method.
	
	The rest of the paper is organized as follows. In Section \ref{sec-pre}, we present a few preliminaries including notations and assumptions.
	In Section \ref{sec-method},  we introduce our two-stage stochastic homogenization method and characterize the piecewise constant approximate coefficients.
	In Section \ref{sec:convergence}, we present the convergence analysis of the proposed two-stage
	stochastic homogenization  method under the stationary assumption on the diffusion coefficients.
	In Section \ref{sec-alg}, we propose a finite element discretization and a detailed
	implementation algorithm for the proposed method.
	In Section \ref{sec:experiments}, we present several benchmark numerical experiments to demonstrate the
	efficiency of the proposed method and to validate the theoretical results.
	Finally, the paper is completed with some concluding remarks given in Section \ref{sec-con}.

	\section{Preliminaries}\label{sec-pre}
	\subsection{Notations and assumptions}
	Standard notations will be adopted in paper. $(\Omega,\,{\cal F},\mathbb{P})$ denotes a probability space and
	${\mathbb E}(\mathbf{X}):=\int_\Omega \mathbf{X}(\omega) \textnormal d \mathbb{P}(\omega)$ stands for the expectation value of random variable $\mathbf{X}\in L^1(\Omega, \textnormal d \mathbb{P})$.
	Let $Q:=(0,1)^d$ be the unit cell and $Q+\mathbf{k} :=(k_0,k_0+1)\times(k_1,k_1+1)\times \cdots\times(k_d,k_d+1)$ for $\mathbf{k}=(k_0,k_1,\cdots,k_d)^T$ with $k_i\in\mathbb{Z}$.
	{\color{black}Let} ${D}\subset \mathbb{R}^d$ be a bounded domain which can be written as ${D}=\cup_{ \mathbf{k}\in \mathbb{Z}^d} {D}_{\mathbf{k}}$,
	where ${D}_{\mathbf{k}}={D}\cap \varepsilon (Q+\mathbf{k})$. For a positive integer
	$M\in \mathbb{Z}^+$, set ${\cal Q}_M:=M Q=(0,M)^d$
	and ${\cal Q}_M^{\mathbf{k}}=MQ+M\mathbf{k}$.  
	Let ${\cal M}(\alpha,\beta; {D})$
	denote the set of invertible real-valued $d\times d$ matrices $ {A}= {A}(\cdot, \omega)$ with entries in $L^\infty ({D})$ and satisfying $\mathbb{P}$-a.s
	\begin{equation}
	\alpha |\xi|^2\leq ( {A}\xi,\xi)\leq \beta  |\xi|^2\qquad\mbox{for any $\xi\in \mathbb{R}^d$ and a.e. 	in ${D}$}.
	\end{equation}
	Here $(\cdot,\cdot)$ denotes the standard inner product in $\mathbb{R}^d$ and $|\xi|^2=(\xi,\xi)$.
	
	{\color{black} Similar to \cite{anantharaman2012introduction, anantharaman2011numerical}}, we also assume that $A(\frac x \varepsilon, \omega)\in {\cal M}(\alpha,\beta; {D})$ is stationary
	in the sense that 
	\begin{equation}\label{stationary}
	A\Bigl(\frac{x}\varepsilon +\mathbf{k},\omega\Bigr)= A\Bigl(\frac{x}\varepsilon , \tau_{\mathbf k} \omega \Bigr)
	\qquad \mbox{for any $\mathbf{k}\in \mathbb{Z}^d$, a.e. in ${D}$ and $\mathbb{P}$-a.s.},
	\end{equation}
	where $\tau_{\mathbf k}\color{black}$ is a mapping which is ergodic and preserves the measure $\mathbb{P}$, that is
	\begin{equation}\label{ergodic}
	 \tau_{\mathbf k}\color{black} E=E \quad\forall E\in \mathcal{F} \quad \mbox{implies that}\quad  \mathbb{P}(E)=0~ \mbox{or} ~1.
	\end{equation}

	For the ease of presentation, we set $f(x,\omega)=f(x)$ in the rest of the paper, that is, $f$ is a deterministic function.
	
	\subsection{Elements of the classical stochastic homogenization theory}\label{sec-2.2}
	It is well known that \cite{anantharaman2012introduction, blanc2007stochastic,
		bourgeat2004approximations, kozlov1979averaging, papanicolaou1979boundary, sab1992homogenization} as $\varepsilon\rightarrow 0$, the solution $u^\varepsilon(x,\omega)$ of equation \eqref{eq:problem} converges to the solution of the following homogenized problem:
	\begin{subequations}\label{eq:homo-problem}
		\begin{alignat}{2}
		-\textnormal{div}\bigl(A^\star\nabla u^\star(x) \bigr)&= f(x) \color{black}&&\qquad \mbox{in }   D, \\
		u^\star(x) &= 0  &&\qquad \mbox{on }\partial   D,
		\end{alignat}
	\end{subequations}
	where   the  $(i,\,j)$ entry \color{black} of the homogenized matrix (or effective coefficient) $A^\star=(a_{ij}^\star)_{d\times d}$ is defined by
	\begin{equation}\label{eq:home-matrix}
	a_{ij}^\star=\mathbb{E}\Big(\int_Q(e_i+\nabla \mathbb{N}_{e_i}(y,\omega))^TA(y,\omega)e_j\textnormal d y\Big).
	\end{equation}
	$\{e_i\}_{i=1}^d$ denotes the canonical basis of $\mathbb{R}^d$ and the cell function $\mathbb{N}_{e_i}(y,\omega)$ is defined as the solution of the following cell problem:
	\begin{subequations}\label{eq:cell-problem}
		\begin{alignat}{2}
		-\textnormal{div}\bigl[A(y,\omega)(e_i+\nabla \mathbb{N}_{e_i}(y,\omega))\bigr] &=0 \qquad \mbox{in}~\mathbb{R}^d, \\
		\mathbb{E}\Big(\int_Q\nabla \mathbb{N}_{e_i}(y,\cdot)\textnormal{d}y\Big) &=0, \\
		\nabla \mathbb{N}_{e_i}(y,\omega)~\textnormal{is~stationary~in}&\textnormal{~the~sense~of~\eqref{stationary}}.
		\end{alignat}
	\end{subequations}
	
	As shown above, the classical stochastic homogenization method obtains the homogenized solution $u^\star(x)$ in one step, {\color{black}see \cref{fig:Fig1}} for a schematic explanation.
	We note that the cell problem \eqref{eq:cell-problem} is random elliptic problem which is  posed on the whole space $\mathbb{R}^d$ and can not be reduced to a cell problem on a bounded domain due to the global constraint $\mathbb{E}\big(\int_Q\nabla  \mathbb{N}_{e_i}(y,\cdot)\textnormal{d}y\big)=0$.
	Solving problem \eqref{eq:cell-problem} is the main computational challenge for
	implementing the classical stochastic homogenization method.  A natural and widely
	used approach (cf. \cite{anantharaman2012introduction,gloria2012optimal}) is to approximate $\mathbb{R}^d$ by a truncated cubic domain ${\cal Q}_{{N}}\subset \mathbb{R}^d$ with size $N^d$ by using ``periodization"
	and ``cut-off" techniques and then to solve the truncated problem
	\begin{subequations}\label{eq:cell-problem-truncated}
		\begin{alignat}{2}
		-\textnormal{div}\bigl[A(y,\omega)(e_i+\nabla \mathbb{N}_{e_i,N}(y,\omega))\bigr] =0 & &&\qquad\textnormal{in } \mathcal{Q}_N, \\
		\mathbb{N}_{e_i,N}(y,\omega)~\hbox{is}~{\cal Q}_N\hbox{-periodic}. & &&
		\end{alignat}
	\end{subequations}
	Consequently, the deterministic homogenized coefficient matrix $A^\star$ can be practically approximated by a random matrix
	$A^\star_N= (a_{ij,N}^\star(\omega))_{d\times d}\color{black}$ whose $(i,j)$ entry is defined as
	\begin{equation}\label{eq:approx-homo-coeff}
	 a_{ij,N}^\star(\omega)\color{black}=\frac{1}{|{\cal Q}_N|}\Big(\int_{{\cal Q}_N}(e_i+\nabla \mathbb{N}_{e_i,N}(y,\omega))^TA(y,\omega)(e_j+\nabla \mathbb{N}_{e_j,N}(y,\omega))\textnormal d y\Big).
	\end{equation}
	Then, the solution $u^*$ of problem \eqref{eq:homo-problem} is approximated as
	$\mathbb{E}[u^*_N(\omega)]$ with $u^*_N(\omega)$ being the solution of the following equation
	\begin{subequations}\label{eq:homo-problem_approx}
		\begin{alignat}{2}
		-\textnormal{div}\bigl(A^\star_N(\omega)\nabla u^\star_N(x,\omega) \bigr)&=f(x)  &&\qquad \mbox{in } D, \\
		u^\star_N(x,\omega) &= 0  &&\qquad \mbox{on }\partial D.
		\end{alignat}
	\end{subequations}
	
	We refer the reader to \cite{anantharaman2012introduction,anantharaman2012elements,gloria2012optimal} for
	a detailed account about the above classical numerical approach.

	\section{ Two-stage stochastic homogenization method}\label{sec-method}
	\label{sec:method}
	{\color{black}In this section,} we shall present a detailed formulation of our two-stage stochastic homogenization method for problem \eqref{eq:problem}.
	The new method only needs to solve similar diffusion equations with constant
	diffusion coefficients {and random right-hand sides}.
		
	\subsection{Formulation of the two-stage stochastic homogenization method}
	As explained earlier, the main difficulty for solving problem \eqref{eq:problem}
	is due to the oscillatory nature of its solution which is caused by the
	oscillatory coefficient matrix $A$ of the problem. Recall that the classical
	numerical homogenization methods approximate effective (or homogenized) coefficient matrix $A^*$ by matrix $A^*_N$ which is formed by solving the cell problem \eqref{eq:cell-problem-truncated}, which is often expensive to solve numerically.
	Motivated by this difficulty, the main idea of our method is
	to propose a different procedure to construct an approximation to $A^*$
	(in the first stage) whose corresponding homogenized problem  can be solved efficiently (in the second stage).
	
	Specifically, our proposed method aims to construct a homogenized solution
	$u_0^0(x)$ by the following two stages as illustrated in \cref{fig:Fig1}.
	In the first stage, for each given sample $\omega$, the composite material with micro-structure is equivalently transformed to a piecewise homogeneous material with coefficient matrix
	$\hat A(x,\omega)=\hat{A}^\mathbf{k}(\omega)=\big(\hat a_{ij}^\mathbf{k}(\omega)\big)$ (see  \cref{Fig2}), {\color{black}referred to as the {\em equivalent matrix} in each block $ {D}\cap\varepsilon{\cal Q}_M^{\mathbf k}$}, where
	\begin{equation}\label{eq:homo-matrix-1}
	\hat a_{ij}^\mathbf{k}(\omega)=\frac1 {|{\cal Q}_M|}\int_{{\cal Q}_M}(e_i+\nabla \mathbb{N}_{e_i}^\mathbf{k}(y,\omega))^TA(y+M\mathbf{k},\omega)(e_j+\nabla \mathbb{N}_{e_j}^\mathbf{k}(y,\omega))\textnormal d y,
	\end{equation}
	and the cell function $\mathbb{N}_{e_i}^\mathbf{k}(y,\omega)$ is defined   as the solution of  the following cell problem on block ${\cal Q}_M^{\mathbf{k}}$:
	\begin{subequations}\label{newcellproblem}
		\begin{alignat}{2}
		 -\textnormal{div}\bigl[A(y,\omega)(e_i+\nabla \mathbb{N}_{e_i}^\mathbf{k}(y,\omega))\bigr] =0& &&\qquad \textnormal{~in}~{\cal Q}_M^{\mathbf{k}}, \\
		\mathbb{N}_{e_i}^\mathbf{k}(y,\omega)~\hbox{is}~{\cal Q}_M^{\mathbf{k}} \hbox{-periodic}. & &&
		\end{alignat}
	\end{subequations}
	Here $M$ is a parameter used to balance the efficiency and accuracy of the proposed method.
 	In numerical simulations, we usually choose $M={\cal O}(1)$.
	Notice that the  { equivalent matrix}  $\hat A(x,\omega)$ in each block ${\cal Q}_M^{\mathbf{k}}$ is a constant matrix.
	The {equivalent matrix} $\hat A(x,\omega)$ can be regarded as a
	coarse-grained approximation of the original matrix $A(\frac x\varepsilon,\omega)$.
	In the coarsening process, the equivalent material with coefficient matrix $\hat A(x,\omega)$ is homogeneous  in each block  ${\cal Q}_M^{\mathbf{k}}$,
	but still maintain the heterogeneity between different blocks
	(see \cref{Fig2}-(b)).
 	It should be pointed out that the equivalent matrix $\hat A(x,\omega)$
	is usually different from $A^\star_N$  obtained by ``periodization" procedure.
	In fact, by comparing cell problems \eqref{newcellproblem} and \eqref{eq:cell-problem-truncated} with $M=N$, it is easy to see that $A^\star_N$ is the same as $\hat A^0(\omega)$  and may be different from $\hat A^{\mathbf k}(\omega)$ for $\mathbf k\neq 0$ due to the possible
	heterogeneity between different blocks (recall that $\hat A^{\mathbf k}(\omega)$ denotes the equivalent matrix in block ${D}\cap\varepsilon {\cal Q}^{\mathbf k}_M$).

	\begin{figure}[t]
		\begin{center}
			\includegraphics[width=0.8\textwidth]{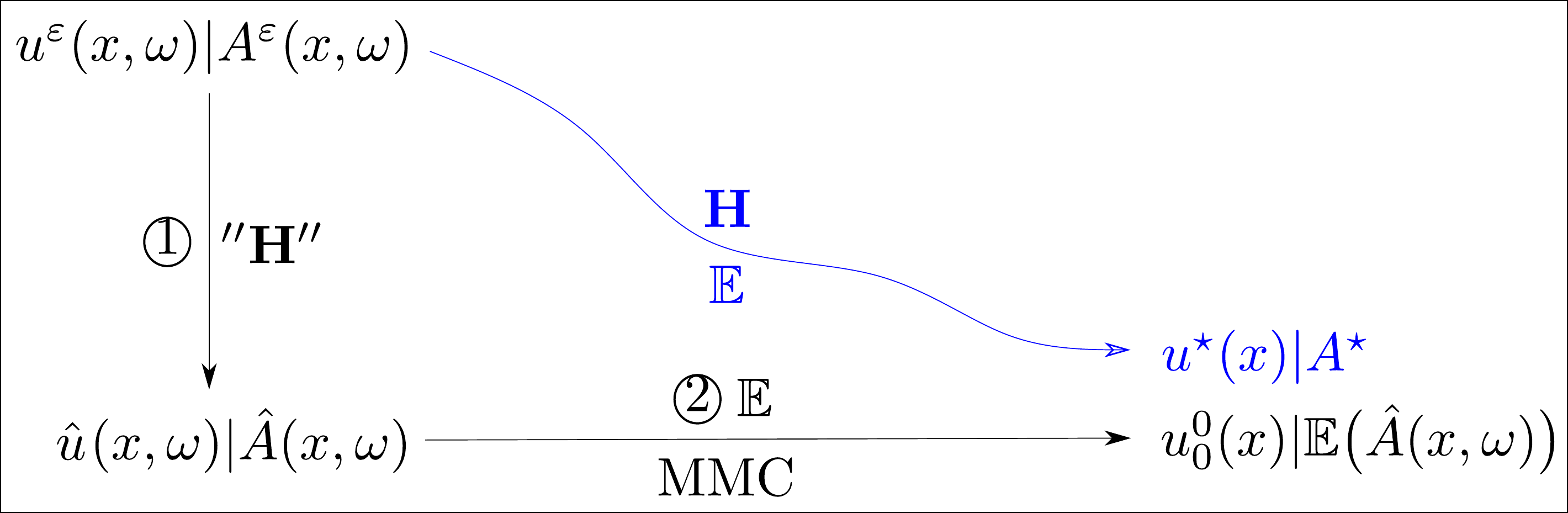}
		\end{center}
		\caption{A schematic diagram for the two-stage stochastic homogenization method.
			In the first stage, the composite material with coefficient matrix $A(\frac x \varepsilon,\omega)$
			is equivalent to a piecewise homogeneous random material with coefficient matrix $\hat A(x,\omega)$.
			In the second stage, the homogenized solution $u_0^0(x)$ is obtained by using the MMC method.
			As a comparison, the classical stochastic homogenization method aims to get the stochastic homogenized solution
			$u^\star(x)$ in one step.}\label{fig:Fig1}
	\end{figure}
	\begin{figure}[h]
		\centering
		{\tiny(a)}\includegraphics[width=0.28\textwidth]{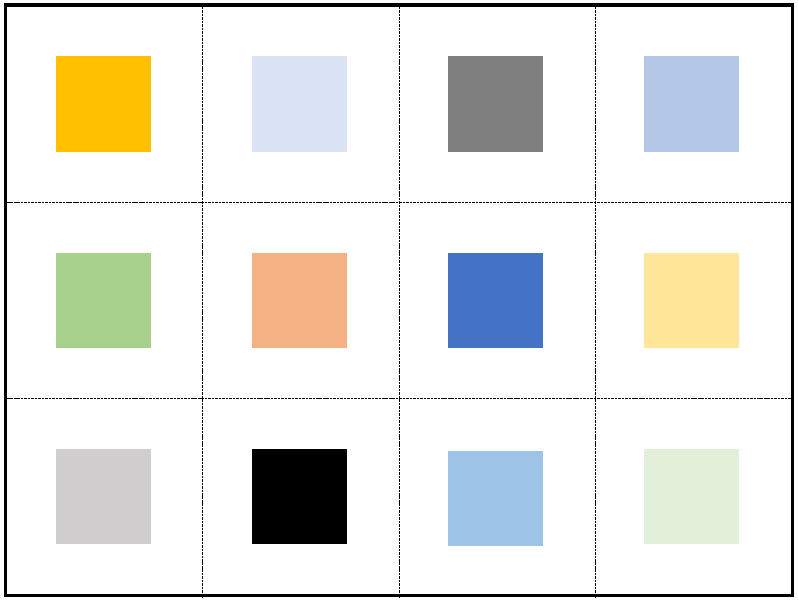}~
		{\tiny(b)}\includegraphics[width=0.28\textwidth]{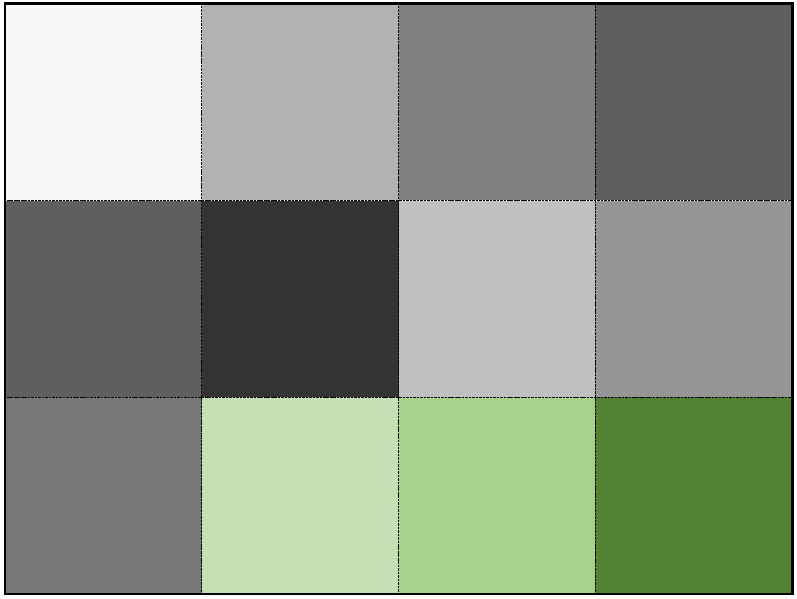}
		{\tiny(c)}\includegraphics[width=0.28\textwidth]{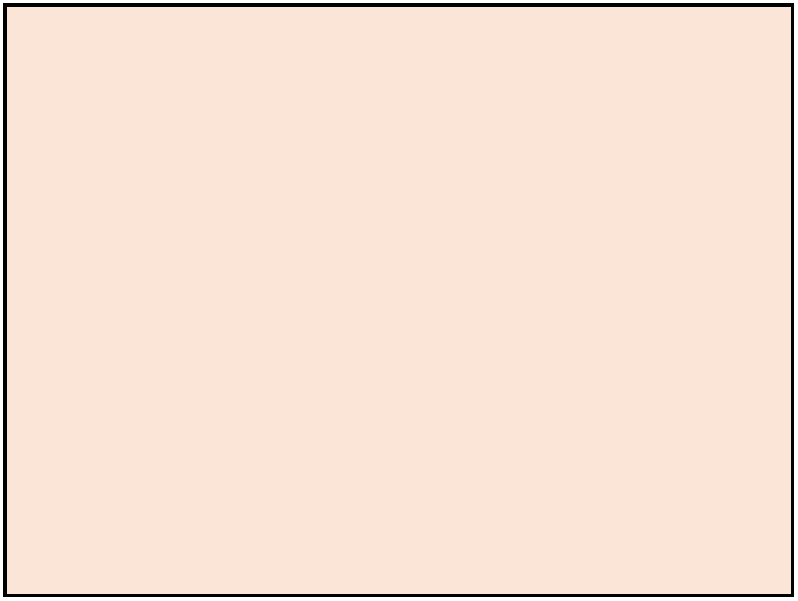}~
		\caption{ (a) Composite material with random coefficients $A(\frac x \varepsilon,\omega)$ for given $\omega$.
			(b) Equivalent material with random coefficients $\hat A(x,\omega)$ for given $\omega$, which
			is a constant matrix in each cell. (c) Stochastic homogenization material with deterministic coefficient $\mathbb{E}(\hat A(x,\omega))$.
		}\label{fig:domain}\label{Fig2}
	\end{figure}
	
In the second stage, we intend to solve the random diffusion problem with the equivalent (piecewise constant) coefficient matrix $\hat A(x,\omega)$, namely,
	\begin{subequations}\label{eq:homo-problem-1}
		\begin{alignat}{2}
		-\textnormal{div}(\hat A(x,\omega)\nabla \hat{u}(x,\omega)) &=f(x)
		&&\qquad \textnormal{in }   D, \\
		\hat{u}(x,\omega) &= 0 &&\qquad \textnormal{on } \partial  D.
		\end{alignat}
	\end{subequations}
{\color{black}
In other words,  the original oscillatory random coefficient matrix
$A(\frac x \varepsilon, \omega)$ is approximated by the equivalent
matrix $\hat A(x,\omega)$ which is a constant matrix
$\hat {A}^{\mathbf k}(\omega)$ in each block $ {D}\cap\varepsilon{\cal Q}_M^{\mathbf k}$.
For a given $\omega$, the computational cost for solving the
homogenized problem
\eqref{eq:homo-problem-1} is less than that
for solving the original problem \eqref{eq:problem}.
However, the equivalent matrix $\hat {A}(x,\omega)$ fluctuates on different blocks
${D}\cap\varepsilon{\cal Q}_M^{\mathbf k}$ due to the non-periodicity.
The fluctuation leads to expensive computational costs for solving
the homogenized problem
\eqref{eq:homo-problem-1} with small parameter $\varepsilon$
and $M={\cal O}(1)$, because the computational mesh size must
be proportional to $M\varepsilon$.
To overcome the difficulty, we adapt the MMC finite element method of \cite{feng2016multimodes} to solve \eqref{eq:homo-problem-1} in an efficient way.
This is possible thanks to our discovery which shows that the equivalent matrix
$\hat A(x,\omega)$ has a nice structure, that is, it can be rewritten as a small
random perturbation of the  deterministic matrix $\mathbb{E}(\hat A(x,\cdot))$, see Section \ref{sec-3.2}. This then sets an ideal stage for us to solve problem \eqref{eq:homo-problem-1} by using the MMC method. The leading term in the MMC approximation will be defined as $u^0_0(x)$, which is the after-sought
approximate solution alluded earlier, see Section \ref{sec-3.3}.

It is important to point out that the MMC  method presented in \cite{feng2016multimodes} can not be directly applied to
a random diffusion problem \eqref{eq:problem} because its diffusion
coefficient does not satisfy the weak media assumption of the MMC method (see Appendix A).

	\subsection{Characterization of the equivalent coefficient matrix $\hat A(x,\omega)$}\label{sec-3.2}
	In this subsection, we show that the equivalent matrix $\hat A(x,\omega)$ can be rewritten as a (small) random perturbation of a deterministic matrix. The precise statement is
	given in the following theorem.
	
	\begin{theorem}\label{smallrandompertur}
		Suppose $A(\frac x\varepsilon,\omega)$ satisfies stationary hypothesis \eqref{stationary}. Then the equivalent matrix  $\hat A(x,\omega)$ can be rewritten as:
		\begin{equation}\label{eq:KL3-1}
		{\hat A}(x,\omega)=\mathbb{E}({\hat A}^0(\omega))+\delta A_1(x,\omega),
		\end{equation}
		where {\color{black}${\hat A}^0(\omega)$ denotes the equivalent matrix in any block ${D}\cap\varepsilon{\cal Q}_M^{\mathbf k}$}, $A_1(x,\omega)=(a^1_{ij}(x,\omega))$ with $a^1_{ij}\in L^2(\Omega,L^\infty( {D}))$, and $\delta$ is a (small) parameter
		{which depends on $\varepsilon$ and $M$.}
	\end{theorem}
	
	\begin{proof}
		By the stationary assumption \eqref{stationary},
		the $(i,j)$-entry of  $\hat A(x,\omega)$ can be written as
		\begin{align}\label{eq:homo-matrix-3}
		&\hat a_{ij}^\mathbf{k}(\omega)=\frac1 {|{\cal Q}_M|}\int_{{\cal Q}_M}(e_i+\nabla \mathbb{N}_{e_i}^\mathbf{k}(y,\omega))^TA(y+M\mathbf{k},\omega)(e_j+\nabla \mathbb{N}_{e_j}^\mathbf{k}(y,\omega))\textnormal \,dy,\\
		&\hskip 0,25in
		=\frac1 {|{\cal Q}_M|}\int_{{\cal Q}_M}(e_i+\nabla \mathbb{N}_{e_i}^\mathbf{k}(y,\tau_{M\mathbf k}\omega))^TA(y,\tau_{M\mathbf k}\omega)(e_j+\nabla \mathbb{N}_{e_j}^\mathbf{k}(y,\tau_{M\mathbf k}\omega))\textnormal \,dy, \nonumber
		\end{align}
		where  the cell function $\mathbb{N}_{e_i}^\mathbf{k}(y,\tau_{M\mathbf k}\omega)$ satisfies
		\begin{subequations}
			\begin{alignat}{2}
			-\textnormal{div}[A(y,\tau_{M\mathbf k}\omega)(e_i+\nabla \mathbb{N}_{e_i}^\mathbf{k}(y,\tau_{M\mathbf k}\omega))] =0& &&\qquad \textnormal{~in}~{\cal Q}_M, \\
			\mathbb{N}_{e_i}^\mathbf{k}(y,\tau_{M\mathbf k}\omega)~\hbox{is}~{\cal Q}_M\hbox{-periodic}.& &&
			\end{alignat}
		\end{subequations}
		
		Thus, $\hat A^{\mathbf k}(\omega)=\hat A^0(\tau_{M\mathbf k}\omega)$, which shows that the equivalent matrix { coincides
		with $A^\star_M(\omega)$ on each block and at each sample.}
		Since the ergodic mapping $\tau_k$ preserves the measure $\mathbb P$, then we have
		\begin{equation}
		\mathbb {E}(\hat{A}^\mathbf{k}(\omega))=\int\limits_\Omega\hat{A}^\mathbf{k}(\omega)\hbox{d}\mathbb P(\omega)
		=\int\limits_\Omega\hat{A}^0(\tau_{M\mathbf k}\omega)\hbox{d}\mathbb P(\tau_{M\mathbf k}\omega)=\mathbb {E}(\hat{A}^0(\omega)),
		\end{equation}
		and
		\begin{align}\label{eq:add-new1}
		\mathbb {V}\hbox{ar}(\hat{A}^\mathbf{k}(\omega))&=\int\limits_\Omega\left(\hat{A}^\mathbf{k}(\omega)-\mathbb E\left(\hat{A}^\mathbf{k}(\omega)\right)\right)^2\hbox{d}\mathbb P(\omega),\\
		&=\int\limits_\Omega\left(\hat{A}^0(\tau_{M\mathbf k}\omega)-\mathbb E\left(\hat{A}^0(\tau_{M\mathbf k}\omega)\right)\right)^2\hbox{d}\mathbb P(\tau_{M\mathbf k}\omega)=\mathbb {V}\hbox{ar}(\hat{A}^\mathbf{0}(\omega)). \nonumber
		\end{align}
		
		Next, we derive a (small) random perturbation form for the equivalent matrix $\hat {A}(x,\omega)$.
		For a given $\mathbf{k}$, $\hat{A}^{\mathbf{k}}(\omega)=(\hat a_{ij}^{\mathbf{k}}(\omega))\in
		L^2(\varepsilon{\cal Q}_M^{\mathbf {k}})$ and the autocorrelation
		function of $\hat a_{ij}^{\mathbf{k}}(\omega)$, which is a constant, is defined by
		\begin{equation}
		{\hbox{Cov}_{\mathbf{k}}} =\mathbb {V}\hbox{ar}(\hat{A}^\mathbf{k}(\omega))=\mathbb {V}\hbox{ar}(\hat{A}^\mathbf{0}(\omega)).
		\end{equation}
		
		Introduce the self-adjoint covariant operator ${\cal T}_{\mathbf{k}}: L^2(\varepsilon{\cal Q}_M^{\mathbf {k}})\rightarrow L^2(\varepsilon{\cal Q}_M^{\mathbf {k}})$ as
		\begin{equation}
		{\cal T}_{\mathbf{k}} v(\cdot):=\int_{\varepsilon{\cal Q}_M^{\mathbf {k}}} {\hbox{Cov}_{\mathbf{k}}}  v(x) dx=
		{\hbox{Cov}_{\mathbf{k}}} \int_{\varepsilon{\cal Q}_M^{\mathbf {k}}}  v(x) dx \quad \forall v\in  L^2(\varepsilon{\cal Q}_M^{\mathbf {k}}).
		\end{equation}
		Let $\{(\lambda_l, \varphi_l)\}_{l\geq 1}$ denote a complete eigen-set of the  operator ${\cal T}_{\mathbf{k}}$ with $\lambda_1> \lambda_2= \cdots =0$ and
		$$ \int_{\varepsilon{\cal Q}_M^{\mathbf {k}}}\varphi_l(x)\varphi_m(x)\hbox{d}x=\delta_{lm}
		\qquad l,\,m=1,\,2,\,\cdots.$$
		
		By the Karhunen-Lo\'{e}ve expansion, we obtain
		\begin{equation}
		\hat a_{ij}^{\mathbf{k}}(\omega)=\mathbb{E} (  \hat a_{ij}^{\mathbf{k}}(\omega))+{ \sum\limits_{l=1}^\infty}
		\sqrt{\lambda_{l}}\mathbf{Z}^{\mathbf{k}}_{l}(\omega)\varphi_{l}(x)=\mathbb{E} (  \hat a_{ij}^{\mathbf{k}}(\omega))+
		\sqrt{\lambda_1}\mathbf{Z}^{\mathbf{k}}_1(\omega)\varphi_1(x),
		\end{equation}
		where $\varphi_1(x)=\left|{\varepsilon{\cal Q}_M^{\mathbf {k}}}\right|^{-\frac12}$ and $\mathbf{Z}^{\mathbf{k}}_1(\omega)$ is a standard normal variable given by
		\begin{equation}
		\mathbf{Z}^{\mathbf{k}}_1(\omega)=\frac { \hat a_{ij}^{\mathbf{k}}(\omega)-\mathbb{E} (  \hat a_{ij}^{\mathbf{k}}(\omega))} {\sqrt{\lambda_1}}\int_{\varepsilon{\cal Q}_{\mathbf {1}}}\varphi_1(x)\hbox{d} x=\sqrt{  \left|{\varepsilon{\cal Q}_{\mathbf {k}}}\right|}\frac { \hat a_{ij}^{\mathbf{k}}(\omega)-\mathbb{E} (  \hat a_{ij}^{\mathbf{k}}(\omega))} {\sqrt{\lambda_1}}.
		\end{equation}
		The principal eigenvalue $\lambda_1$ satisfies
		\begin{equation}
		\begin{aligned}
		\lambda_1&=\frac{\int_{\varepsilon{\cal Q}_M^{\mathbf {k}}} {\cal T}_{\mathbf{k}} \varphi_1 \cdot\varphi_1(x)\hbox{d}x}{\int_{\varepsilon{\cal Q}_M^{\mathbf {k}}}  \varphi_1 (x)\varphi_1(x)\hbox{d}x}={\hbox{Cov}_{\mathbf{k}}}\left(\int_{\varepsilon{\cal Q}_M^{\mathbf {k}}}  \varphi_1(x)\hbox{d}x\right)^2=
		\left|{\varepsilon{\cal Q}_M^{\mathbf {k}}}\right| {\hbox{Cov}_{\mathbf{k}}},
		\end{aligned}
		\end{equation}
		which implies that
		\begin{equation}
		\hat a_{ij}^{\mathbf{k}}(\omega)=\mathbb{E} (  \hat a_{ij}^{\mathbf{k}}(\omega))+
		\sqrt{{\hbox{Cov}_{\mathbf{k}}}}\mathbf{Z}^{\mathbf{k}}_1(\omega)=\mathbb{E} (  \hat a_{ij}^{\mathbf{0}}(\omega))+
		\sqrt{\mathbb {V}\hbox{ar}(\hat{a}_{ij}^\mathbf{0}(\omega))}\mathbf{Z}^{\mathbf{k}}_1(\omega).
		\end{equation}
		Thus, $(i,j)$-component of $\hat A(x,\omega)$ satisfies
		\begin{equation}
		\hat a_{ij}(x,\omega)=\mathbb{E} (  \hat a_{ij}^{\mathbf{0}}(\omega))+
		\sqrt{\mathbb {V}\hbox{ar}(\hat{a}_{ij}^\mathbf{0}(\omega))}\sum\limits_{\mathbf{k}}\mathbf{Z}^{\mathbf{k}}_1(\omega)\Xi_{\mathbf k}(x).
		\end{equation}
		where $\Xi_{\mathbf k}(x)$ stands for the characteristic function of the
		domain $ {D}\cap \varepsilon{\cal Q}_{\mathbf k}$.
		
		Finally, setting $\delta := \max_{1\leq i,j\leq d} \sqrt{\mathbb {V}\hbox{ar}(\hat{a}_{ij}^\mathbf{0}(\omega))}$, then
		the equivalent matrix $\hat{A}(x,\omega)$ can be rewritten as a (small) random perturbation as stated in \eqref{eq:KL3-1}.
	\end{proof}

	\begin{remark}
The smallness of $\delta$ and the precise relationship between $\delta$ and $M$ as well as $\varepsilon$ will be given in \cref{theorem:new-3.2} later in the next section.
	\end{remark}


\subsection{An efficient multi-modes Monte Carlo method for solving the homogenized
	problem \eqref{eq:homo-problem-1}} \label{sec-3.3}

By \cref{smallrandompertur} we know that the equivalent matrix  $\hat{A}(x,\omega)$ can be rewritten as a (small) random perturbation of $\mathbb{E}({\hat A}^0(\omega))$
as given in \eqref{eq:KL3-1}. 
This then sets the stage for us to solve homogenized problem  \eqref{eq:homo-problem-1} by using the MMC method.
To proceed, we first notice that $A_1(x,\omega)=(a^1_{ij}(x,\omega))$ with $a^1_{ij} \in L^2(\Omega,L^{\infty}({D}))$ satisfying
\begin{equation*}
\mathbb{P}\bigl\{\omega\in\Omega; \left\| a_{ij}^1(\omega)\right\|_{L^{\infty}({D})}\leq \bar{a}\bigr\}=1.
\end{equation*}
In \cite{feng2016multimodes},  the perturbation term is assumed to
be in $ L^2(\Omega,W^{1,\infty}({D}))$.
However, the perturbation term $A_1(x,\omega)$ in \eqref{eq:KL3-1} is not in $ L^2(\Omega,W^{1,\infty}({D}))$ because $A_1(x,\omega)$ is a piecewise constant function which is discontinuous in ${D}$. Nevertheless,
we show below that the MMC method can be easily extended to the case.

{\color{black} Due to the linear nature of the equivalent problem \cref{eq:homo-problem-1}
and the small random perturbation structure of the equivalent matrix  $\hat{A}(x,\omega)$,
we can postulate the following multi-modes expansion for $\hat u(x,\omega)$:}
\begin{equation}\label{eq:multimodes1}
\hat u(x,\omega)={\sum_{n=0}^{\infty}\delta^n u^0_{n}(x,\omega)}.
\end{equation}
Substituting \cref{eq:KL3-1} and \cref{eq:multimodes1} into \cref{eq:homo-problem-1} and matching the coefficients of $\delta^n$ order terms for $n=1,2,\cdots,$  we get
\begin{subequations}\label{eq:multimodes2}
	\begin{align}
	-\nabla \cdot\left( \mathbb{E}[{\hat A}^0(\omega)]\nabla
	 u^0_{0}(x,\omega)\right) &=f(x),\\
	-\nabla \cdot\left( \mathbb{E}[{\hat A}^0(\omega)]\nabla u^0_{n}(x,\omega)\right) &=\nabla \cdot\left( { A}_1(x,\omega)\nabla u^0_{n-1}(x,\omega)\right) \quad \forall n\geq 1,\\
	u^0_n(x,\omega) &=0\quad \hbox{on}\quad\partial{D}\quad \forall n\geq 0.
	\end{align}
\end{subequations}

Clearly, the first mode function $u_0^0(\omega,x)$ satisfies a diffusion equation with a deterministic coefficient matrix $\mathbb{E}({\hat A}^0(\omega))$
and {a deterministic source term $f$}. Thus, $u_0^0(\omega,x)$ is independent of $\omega$ and we relabel it as $u_0^0(x):=u_0^0(\omega,x)$.
Moreover, the mode functions $\{u^0_{n}\}_{n\geq 0}$ satisfy a family of
 diffusion equations that have the same deterministic diffusion
 operator $L_0(\cdot): =-\nabla \cdot\left( \mathbb{E}[{\hat A}^0(\omega)]\nabla
 \cdot \right)$ but different right-hand side source terms. Furthermore, $\{u^0_{n}\}_{n\geq 1}$ are defined recursively with the current mode function
 $u_n$ being only dependent directly on the proceeding mode function $u_{n-1}$.
 The well-posedness of multi-mode functions $\{u^0_n\}$ and the
 corresponding error estimates will be discussed in the next section (see \cref{theorem:exist}  and \cref{theorem:2}).

{\color{black}
In this paper, we shall use the first term of the multi-modes expansion, that is, $u_0^0(x)$  
as an approximation for $\hat{u}(x,\omega)$. How to efficiently compute the mode
functions $\{u_n^0\}_{n\geq1}$ is important and challenging, since $A_1(x,\omega)$ in
the right-hand side of (\ref{eq:multimodes2}b) fluctuates  in probability space
and oscillates in different blocks ${D}\cap\varepsilon {\cal Q}^{\mathbf k}_M$.
A natural but expensive approach is to solve  (\ref{eq:multimodes2}b)
for each sample $\omega$ by using a fine mesh with mesh size proportional
to $M\varepsilon$, which is easily implemented but low efficient.
According to \cref{theorem:2} in the next section,  taking $u_0^0(x)$  
as an approximation of $\hat{u}(x,\omega)$ only enjoys the first order convergence rate.
The convergence rate can be further improved by using more mode
functions $\{u_n^0\}_{n\geq1}$. Thus, developing more efficient numerical methods and algorithms for 
computing the mode functions $\{u_n^0\}_{n\geq1}$ are  important and will be addressed in future work.
}

\section{Convergence analysis}
\label{sec:convergence}
In this section, we analyze the convergence and error estimates for the proposed  two-stage stochastic homogenization method.
We first show the convergence of the equivalent matrix $\hat A(x,\omega)$  to
the homogenization  matrix $A^\star$ as $M\to  \infty$ in the next theorem.

\begin{theorem}\label{theorem:new1} Suppose $A(\frac x\varepsilon,\omega)
	$ satisfies the stationary hypothesis \eqref{stationary}.
	Let  $\hat A(x,\omega) =(\hat a_{ij}(x,\omega))$ be the equivalent matrix defined by \eqref{eq:homo-matrix-1} and $ A^\star =(a_{ij}^\star)$ be defined by \eqref{eq:home-matrix},
	then there holds
	\begin{equation}
	\lim_{M\rightarrow\infty}  \|\hat {a}_{ij}(\cdot,\omega) -a^\star_{ij}\|_{L^\infty(D)}=0  \quad \text{a.s.}.
	\end{equation}
\end{theorem}

\begin{proof}
	By Theorem 1 in \cite{bourgeat2004approximations}, we have
	\begin{equation}
	\lim_{M\rightarrow\infty}  \|\hat {a}^0_{ij}(\cdot,\omega) -a^\star_{ij}\|_{L^\infty(D)}=0  \quad \hbox{a.s.}.
	\end{equation}
	For any $\mathbf{k}$, it follows that
	\begin{equation}
	 \lim_{M\rightarrow\infty} \hat {a}^{\mathbf k}_{ij}(\omega) =\lim_{M\rightarrow\infty} \hat {a}^{0}_{ij}(\tau_{M\mathbf{k}}\omega) =a_{ij}^\star
	\quad {\color{black}\hbox{a.s.}} \mbox{ in } L^\infty(D).
	\end{equation}
The proof is complete.
\end{proof}

To derive the rate of convergence of $\hat {A}(x,\omega)$, we introduce the uniform mixing condition as in \cite{bourgeat2004approximations}.
For a given random field $B(x,\omega)$ in $\mathbb{R}^d$, let ${\cal F}_{\tilde{{D}}}$ denote the $\sigma$-algebra $\sigma\{ B(x),\ x\in {\tilde{{D}}}\}$.
The uniform mixing coefficient of $B$ is defined as
\begin{equation}\label{eq:uniformmixing}
\gamma(s)=\sup\limits_{{\tilde D}_1,{\tilde D}_2\subset\mathbb{R}^d,\ \hbox{dist}({\tilde D}_1,{\tilde D}_2)\geq s}
\quad\sup\limits_{\tilde{{\cal D}}_1\in {\cal F}_{\tilde{{\cal D}}_1},\tilde{{\cal D}}_2\in {\cal F}_{\tilde{{\cal D}}_2}}\left|
\mathbb{P}(\tilde{{\cal D}}_1\cap \tilde{{\cal D}}_2)-\mathbb{P}(\tilde{{\cal D}}_1)\mathbb{P}(\tilde{{\cal D}}_2)
\right|.
\end{equation}

In the rest of this section,
we assume $\gamma(s)$ satisfies the following growth condition:
\begin{equation}\label{eq:uniformmixing-1}
\gamma(s)\leq c(1+s)^{-\theta}\quad \hbox{for~some}~\theta>0~\hbox{and}~ \forall s>0.
\end{equation}
Then we have

\begin{theorem}\label{theorem:new-3.2}
	Suppose $A(\frac x\varepsilon,\omega)$ satisfies the stationary hypothesis \eqref{stationary}.
	Assume that  the uniform mixing coefficient of $A(\frac x\varepsilon,\omega)$ satisfies \eqref{eq:uniformmixing-1}.
Let $\hat A(x,\omega)=(\hat a_{ij}(x,\omega))$ and $ A^\star=(a_{ij}^\star)$,
then there holds
	\begin{equation}\label{eq:theorem2:1}
	\mathbb{E}\left[(\hat a_{ij}(x,\cdot)-a_{ij}^\star)^2\right]\leq C M^{-\zeta} \qquad\mbox{for a.e. } x\in D,
	\end{equation}
	and there exists $\zeta=\zeta(\theta,\alpha,d)>0$ such that
	{
	\begin{subequations}\label{eq:theorem2:2}
	\begin{align}
	\mathbb {V}\textnormal{ar}(\hat{a}_{ij}^0(\omega))&\leq C M^{-\zeta},\\
	 \delta&\leq C M^{-\zeta/2}.
	\end{align}
\end{subequations}}
\end{theorem}

\begin{proof}
	By Theorem 5 of \cite{bourgeat2004approximations}, we have
	\begin{equation}
	\mathbb{E}\left[(\hat a^0_{ij}(\cdot)-a_{ij}^\star)^2\right]\leq C M^{-\zeta} \qquad\mbox{for a.e. } x\in D.
	\end{equation}
	 For any $\mathbf{k}$, similar to \eqref{eq:add-new1}, we have
	\begin{equation}
	\mathbb{E}\left[(\hat a^{\mathbf k}_{ij}(\cdot)-a_{ij}^\star)^2\right]=\mathbb{E}\left[(\hat a^0_{ij}(\cdot)-a_{ij}^\star)^2\right]\leq C M^{-\zeta} \qquad\mbox{for a.e. } x\in D,
	\end{equation}
	which completes the proof of \eqref{eq:theorem2:1}.
	To show \eqref{eq:theorem2:2}, using the H\"{o}lder's inequality and \eqref{eq:theorem2:1}, we get
	\begin{equation}
	\bigl(\mathbb{E}[\hat a^0_{ij}]-a_{ij}^\star \bigr)^2=\left(\mathbb{E}[\hat a^0_{ij}-a_{ij}^\star]\right)^2
	\leq \mathbb{E}[(\hat a^0_{ij}-a_{ij}^\star)^2]\leq C M^{-\zeta},
	\end{equation}
	which implies
	\begin{align}
	\mathbb {V}\textnormal{ar}(\hat{a}_{ij}^0)
	&=\mathbb{E}\left[(\hat a^0_{ij}-\mathbb{E}[\hat a^0_{ij}])^2\right],\\
	&\leq 2\mathbb{E}\left[(\hat a^0_{ij}-a_{ij}^\star)^2\right]+
	2\mathbb{E}\left[(a_{ij}^\star-\mathbb{E}[\hat a^0_{ij}])^2\right], \nonumber\\
	&\leq C M^{-\zeta}. \nonumber
	\end{align}
	Thus, the proof is complete.
\end{proof}

\begin{remark}
  One immediate corollary of \cref{theorem:new-3.2} is that $\delta \rightarrow 0$ as $\varepsilon \rightarrow 0$ after taking $M=  \varepsilon^{-\sigma}$ with $\sigma \in (0,1)$. This verifies that
  $\delta$ is indeed a small parameter.
\end{remark}

Following Theorem 3.1 of \cite{feng2016multimodes}, we can show the unique existence and the stability estimate for each mode $u^0_n$ in \eqref{eq:multimodes1}.

\begin{theorem}  \label{theorem:exist}
	Assume that $A(\frac x \varepsilon,\omega)\in {\cal M}(\alpha,\beta, {D})$ and $f\in L^{2}( {D})$.
	There exists a unique solution $u^0_n\in L^2(\Omega,H^1_0( {D}))$ to problem \eqref{eq:multimodes2} for each $n\geq0$ which satisfies
	\begin{equation}\label{eq:regularity}
	\mathbb{E}(\|u_n^0\|^2_{H^{1}( {D})})\leq C_0^{n+1}\|f\|^2_{L^{2}( {D})}.
	\end{equation}
	 for some $C_0>0$ independent of $n$ and $\delta$.
\end{theorem}

\begin{proof}
	Using Hashin-Shtrikman bounds \cite{hashin1962variational}, we have $\hat A(x,\omega)\in {\cal M}(\alpha,\beta,{D})$.
	For $n=0$, the existence of a unique weak solution $u^0_0\in H^1_0( {D})$ follows immediately from the Lax-Milgram Theorem (see, e.g. \cite{gilbarg2015elliptic}) and there holds
	\begin{equation}
	\mathbb{E}(\|u_0^0\|^2_{H^1( {D})})\leq C_0^1\|f\|^2_{L^{2}( {D})}.
	\end{equation}
	
	The proof for the cases $n\geq 1$ can be done by the induction
	argument. Assume that \eqref{eq:regularity} holds for $n=0,1,\cdots,l-1$, then $A_1(x,\omega)\nabla u^0_{n-1}\in \left[L^2(\Omega, L^2( {D}))\right]^d$ (which is the source term in  \eqref{eq:multimodes2}).
	Using Theorem 3.3 in \cite{oleinik2009mathematical}, there exists
	a unique $u_l^0\in L^2(\Omega,H^1_0(D))$ which solves  \eqref{eq:multimodes2} for $n=l$ and fulfills the following estimate:
	\begin{align}
	\mathbb{E}(\|u_l^0\|^2_{H^1( {D})})&\leq C(D)\mathbb{E}\left(\| A_1(x,\omega)\nabla u^0_{n-1}\|^2_{[L^2( {D})]^d}\right),\\
	&\leq d C(D) \bar{a}^2\mathbb{E}\left(\| \nabla u^0_{n-1}\|^2_{[L^2( {D})]^d}\right), \nonumber\\
	&\leq
	d C(D) \bar{a}^2\mathbb{E}\left(\|  u^0_{n-1}\|^2_{H^1( {D})}\right), \nonumber\\
	&\leq C^{l+1}_0\|f\|^2_{L^{2}({D})}. \nonumber
	\end{align}
	The proof is complete.
\end{proof}

 The next theorem establishes an estimate for the error function $\hat u(x,\omega)-u^0_0(x)$, which is an analogue to Theorem 3.2 of \cite{feng2016multimodes}.

\begin{theorem}\label{theorem:2}
	Suppose that $A(\frac x \varepsilon,\omega)\in {\cal M}(\alpha,\beta, {D})$
	satisfies stationary hypothesis \eqref{stationary}.  Assume that $\delta<1$ and $f\in L^{2}(\Omega)$, we have
	\begin{equation}\label{eq:error1}
	\mathbb{E}\bigl(\left\| \hat u(x,\omega ) - u_0^0(x) \right\|_{{H^{1}}( {D} )}^2\bigr) \leqslant C_1{\delta ^{2}}\left\| {f} \right\|_{L^{2}( {D} )}^2.
	\end{equation}
	for some positive constant $C_1$ independent of $\delta$.
\end{theorem}

\begin{proof}
	Let $r(x,\omega):=\hat u(x,\omega)-u^0_0(x)$. Substituting \eqref{eq:KL3-1} into  \eqref{eq:homo-problem-1} and combining it with the first equation of \cref{eq:multimodes2}, we obtain
	\begin{subequations}
		\begin{alignat}{2}
		-\hbox{div}(\mathbb{E}({\hat A}(x,\omega))\nabla r(x,\omega)) &=\delta \hbox{div}(A_1(x,\omega)\nabla \hat u(x,\omega ))  &&\qquad \hbox{in } {D}, \\
		r(x,\omega)&=0 &&\qquad \hbox{on } \partial {D}.
		\end{alignat}
	\end{subequations}
	By Theorem 3.3 of \cite{oleinik2009mathematical},  we get
	\begin{align*}
	\mathbb{E}\bigl(\left\| r\right\|_{{H^{1}}( {D} )}^2\bigr) &\leqslant {C}(D){\delta ^{2}}\mathbb{E}\bigl(\left\|
	A_1(x,\omega)\nabla {\hat u} \right\|_{[L^2( {D})]^d}^2\bigr),\\
	&\leqslant d{C}(D)\bar{a}{\delta ^{2}}\mathbb{E}\bigl(\left\|
	\nabla {\hat u} \right\|_{[L^2( {D})]^d}^2\bigr), \nonumber\\
	&\leqslant C_1 {\delta^{2}}\left\| f \right\|_{L^{2}( {D} )}^2. \nonumber
	\end{align*}
	which completes the proof.
\end{proof}

\begin{theorem}\label{convergenceresult}
	Let $M=C \varepsilon^{-\sigma}$ with $\sigma \in(0,1)$ and $u^\varepsilon$ denote the solution of problem \eqref{eq:problem}.
	Assume that $f\in L^{2}( {D})$, $A(\frac x \varepsilon,\omega)\in{\cal M}(\alpha,\beta; {D})$ satisfies the stationary hypothesis \eqref{stationary} and is ${\cal Q}_M$-periodic for $\mathbb{P}$-a.s. $\omega\in \Omega$.  Then there holds
	\begin{equation}\label{eq:error1-1}
	\lim_{\varepsilon\rightarrow 0}\mathbb{E}\left(\left\| {{u^\varepsilon}(x,\omega ) - u_0^0(x)} \right\|_{{L^{2}}( {D} )}^2\right)  =0.
	\end{equation}
\end{theorem}

\begin{proof}
	By \cref{theorem:2}, we have
	\begin{equation}
	\mathbb{E}\Bigl(\left\| {{\hat u}(x,\omega ) - u_0^0(x)} \right\|_{{L^2}( {D} )}^2\Bigr) \leq
	C_1{\delta ^{2}}\left\| {f} \right\|_{{L^{2}}( {D} )}^2.
	\end{equation}
	Since  $M=C \varepsilon^{-\sigma}$ with $\sigma \in(0,1)$, it follows from \cref{theorem:new-3.2} that
	\begin{equation}
	\lim_{\varepsilon\rightarrow 0}\mathbb{E}\Bigl(\left\| {{\hat u}(x,\omega ) - u_0^0(x)} \right\|_{{L^2}( {D} )}^2\Bigr) =0,
	\end{equation}
	By Theorem 6.1 of \cite{cioranescu1999introduction}, there hold for   given $\omega\in \Omega$
	\begin{equation}
	u^\varepsilon(x,\omega )\rightharpoonup \hat u(x,\omega)\quad \hbox{weakly~in}~H^1_0( {D})~\hbox{as}~\varepsilon\rightarrow 0,
	\end{equation}
	and
	\begin{equation}
	u^\varepsilon(x,\omega )\rightarrow \hat u(x,\omega)\quad \hbox{strongly~in}~L^2( {D})~\hbox{as}~\varepsilon\rightarrow 0.
	\end{equation}
	By the triangle inequity, we get
	\begin{align}
	\lim_{\varepsilon\rightarrow 0}\mathbb{E}\left(\left\| {{u^\varepsilon}(x,\omega ) - u_0^0(x)} \right\|_{{L^{2}}( {D} )}^2\right)
	&\leqslant
	\lim_{\varepsilon\rightarrow 0}\mathbb{E}\left(\left\| {{u^\varepsilon}(x,\omega ) - \hat u(x,\omega)} \right\|_{{L^{2}}( {D} )}^2\right),\\
	&\quad +\lim_{\varepsilon\rightarrow 0}\mathbb{E}\Bigl(\left\| {{\hat u}(x,\omega ) - u_0^0(x)} \right\|_{{L^2}( {D} )}^2\Bigr) =0.\nonumber
	\end{align}
	which completes the proof.
\end{proof}

We note that the above theorem establishes the convergence of the proposed two-stage stochastic homogenization method in the case when
$A(\frac x \varepsilon,\omega)$ is ${\cal Q}_M$-periodic,
the convergence for the non-periodic case remains open and will be addressed in  future work.

\section{Implementation algorithm for the proposed two-stage method} \label{sec-alg}
In this section, we address the implementation issues for the proposed two-stage
stochastic homogenization method.
Let $\{\omega_s\}_{s=1}^L$ be $L$ independent and identically distributed (i.i.d.) random samples in the sample space $\Omega$.
Let $\mathbb{T}^{\mathbf{k}}_h$ stand for a quasi-uniform partition of
  cell ${\cal Q}^{\mathbf k}_M$ such that $\overline{{\cal Q}_M^{\mathbf k}}=\cup_{K_{\mathbf{k}}\in \mathbb{T}^{\mathbf{k}}_h}\overline {K_{\mathbf{k}}}$.
Here we assume that the partition $\mathbb{T}^{\mathbf{k}}_h$ is a body-fitted grid according to the  coefficients matrix $A(y+M\mathbf{k},\omega_s)$.
Let $\mathbb{V}_{r}^{{\mathbf{k}},h}$ denote the standard finite element space
of degree $r$ defined by
\begin{equation}
\mathbb{V}_{r}^{{\mathbf{k}},h}:=\{v\in H_{per}^1( {\cal Q}^{\mathbf k}_M);   v|_{K_{\mathbf{k}}}~\hbox{is a polynomial of degree}~r~ \hbox{for each}~
K_{\mathbf{k}}\in \mathbb{T}^{\mathbf{k}}_h\}.
\end{equation}
Here $ H_{per}^1({\cal Q}^{\mathbf k}_M)\color{black}=\{v\in  H^1( {\cal Q}^{\mathbf k}_M);  ~\hbox{such that} ~v~\hbox{is}~{ {\cal Q}^{\mathbf k}_M  }\hbox{-periodic}\}$.

The finite element cell solutions
$\mathbb{N}_{e_i}^{\mathbf{k},h}(y,\omega_s)$ are defined as
\begin{equation}\label{eq:new-2.23}
 \left(A(y,\omega_s)(e_i+\nabla \mathbb{N}_{e_i}^{\mathbf{k},h}(y,\omega_s)),\nabla  v^h\right)_{{\cal Q}^{\mathbf k}_M} =0\qquad \forall v^h\in \mathbb{V}_{r}^{{\mathbf{k}},h},
\end{equation}
 where $(u^h,v^h)_{{\cal Q}^{\mathbf k}_M} $ denotes the standard $L^2$-inner product
over ${{\cal Q}^{\mathbf k}_M}$. \color{black}
With a help of  the finite  element  cell solutions  $\mathbb{N}_{e_i}^{\mathbf{k},h}(y,\omega_s)$,
the $(i,j)$-component of the approximate { equivalent matrix} $\hat{A}^h(x,\omega_s)$
in block $ {D}\cap\varepsilon{\cal Q}_M^{\mathbf{k}}$ is given by
\begin{equation}\label{eq:homo-algo-1}
\hat a_{ij}^{\mathbf{k},h}(\omega_s)=\frac1 {|{\cal Q}_M|}\int_{{\cal Q}_M}(e_i+\nabla \mathbb{N}_{e_i}^{\mathbf{k},h}(y,\omega_s))^T
A(y+M\mathbf{k},\omega_s)(e_j+\nabla \mathbb{N}_{e_j}^{\mathbf{k},h}(y,\omega_s))\textnormal d y.
\end{equation}
We define the empirical mean and variance for each component of
$\hat A^{\mathbf{k},h}(\omega_s)$ in block ${D}\cap\varepsilon{\cal Q}_M^{\mathbf{k}}$ as
\begin{subequations}\label{eq:new-2.25}
\begin{align}
\mu^{\mathbf{k}}_{ij,L} &:=\mu_{L} \left(\hat a_{ij}^{\mathbf{k},h}(\omega_s)\right)=\frac 1 {L}\sum\limits_{s=1}^{L}\hat a_{ij}^{\mathbf{k},h}(\omega_s),\\
\sigma^{\mathbf k}_{ij,L}&:=\sigma_{L} \left(\hat a_{ij}^{\mathbf{k},h}(\omega_s)\right)=\frac 1 {L-1}\sum\limits_{s=1}^{L}\left(\hat a_{ij}^{\mathbf{k},h}(\omega_s)-\mu^{\mathbf{k}}_{ij,L}\right)^2.
\end{align}
\end{subequations}

Similar to \cite{anantharaman2012introduction}, by the strong law of large numbers and the fact that $\hat a_{ij}^{\mathbf{k},h}(\omega_s)$ is i.i.d., we have
\begin{equation}
\begin{aligned}
\mu^{\mathbf k}_{ij,L} \xrightarrow{L\rightarrow +\infty} \mathbb{E}\left(\hat a_{ij}^{\mathbf{k},h}(\omega_s)\right) \quad \mathbb{P}\hbox{-a.s.}.
\end{aligned}
\end{equation}
It follows from the central limit theorem that
\begin{equation}
\begin{aligned}
\sqrt{L}\left(\mu^{\mathbf k}_{ij,L} -\mathbb{E}\left(\hat a_{ij}^{\mathbf{k},h}(\omega_s)\right)\right)\xrightarrow{L\rightarrow +\infty}
\sqrt{\mathbb V\hbox{ar}\left(\hat a_{ij}^{\mathbf{k},h}(\omega_s)\right)}\, {\cal N}(0,1),
\end{aligned}
\end{equation}
where the convergence holds in law and ${\cal N}(0,1)$ denotes the standard Gaussian law. For sufficiently large $L$, we use $\mu^{\mathbf k}_{ij,L} $ as
 an approximation of $\mathbb{E}\left(\hat a_{ij}^{\mathbf{k},h}(\omega_s)\right)$.

In the finite element approximation of the proposed two-stage stochastic homogenization method, we use $\mu_L=(\mu^{0}_{ij,L})$ as an approximation of
$\mathbb{E}\left[\hat A^0(\omega)\right]$.
Let $\mathbb{T}_{h_0}$ be a quasi-uniform partition of computational domain ${D}$ with mesh size $h_0$ such that $\overline{{D}}=\cup_{K\in \mathbb{T}_{h_0}}\overline {K}$.
Assume $\mathbb{V}_{r}^{h_0}$ be the standard finite element space of order
$r$ over $\mathbb{T}_{h_0}$ defined by
\begin{equation}
\mathbb{V}_{r}^{h_0}:=\{v\in H_{0}^1(D); v|_{K}~\hbox{is a polynomial of degree}~r~ \hbox{for each}~
K\in \mathbb{T}_{h_0}\}.
\end{equation}
The finite element approximation of the first mode function $u^0_0$ is defined by
\begin{equation}\label{eq:new-2.29}
\left(\mu_L \nabla u_0^{0,h_0}(x),\nabla  v^{h_0}\right)_{D} =(f,v^{h_0})_{D}\qquad \forall v^{h_0}\in \mathbb{V}_{r}^{h_0}.
\end{equation}
Its implementation algorithm is given below in \cref{alg:buildtree1}.

\begin{algorithm}[h]
	\caption{ The finite element two-stage stochastic homogenization method}
	\label{alg:buildtree1}
	\begin{algorithmic}[1]
		\STATE{Generate a family of i.i.d. samples $\{\omega_{s}\}_{s=1}^L$. For each $\omega_s$, construct a body-fitted and quasi-uniform partition of the cell ${\cal Q}_M$ according to  $A(y,\omega_s)$ for $y\in{\cal Q}_M$. }
		\STATE{Compute finite element cell solutions
			$\mathbb{N}_{e_i}^{\mathbf{k},h}(y,\omega_s)$ with $\mathbf k =0$ by solving cell problem \eqref{eq:new-2.23} and $(i,j)$-entry of the equivalent matrix $\hat A^{0,h}(\omega_s)$ by \eqref{eq:homo-algo-1}. }
		\STATE{Calculate the $(i,j)$-component of the empirical mean matrix $\mu_L$ by using \eqref{eq:new-2.25}.}
		\STATE{Solve the finite element equation \cref{eq:new-2.29} and
			obtain the  finite element approximation of the first mode function $u^0_0$, which is taken as  the
			finite element two-stage stochastic homogenization solution for the random diffusion equation \eqref{eq:problem}. }
	\end{algorithmic}
\end{algorithm}

Since direct simulations of the random diffusion equation \eqref{eq:problem}
are computationally expensive, we use the empirical mean
$ \mathbb{E}(\hat u(x,\omega))$
of the equivalent solution for \eqref{eq:homo-problem-1} as a reference solution
to verify the efficiency and accuracy of the finite element two-stage stochastic homogenization method.
To the end, let $\mathbb{T}_{h_1}$ be  a quasi-uniform partition of computational domain ${D}$ with mesh size $h_1$
such that $\overline{{D}}=\cup_{K\in \mathbb{T}_{h_1}}\overline {K}$,
and $\mathbb{V}_{r}^{h_1}$ be the standard finite element space of
order $r$ defined by
\begin{equation}
\mathbb{V}_{r}^{h_1}:=\{v\in H_{0}^1(D); v|_{K}~\hbox{is a polynomial of degree}~r~ \hbox{for each}~
K\in \mathbb{T}_{h_1}\}.
\end{equation}
Let $ \hat u^{h_1}(x,\omega_s)$ be the solution of \eqref{eq:homo-problem-1} with $\omega=\omega_s$ defined by
\begin{equation}\label{eq:new-2.29-1}
\left(\hat A^{\mathbf{k},h}(\omega_s) \nabla \hat u^{h_1}(x,\omega_s)  ,\nabla  v^{h_1}\right)_{D} =(f,v^{h_1})_{D}\qquad \forall v^{h_0}\in \mathbb{V}_{r}^{h_1}.
\end{equation}
Its implementation algorithm is given below in \cref{alg:buildtree2}.

\begin{algorithm}
	\caption{Algorithm for computing the reference solution}
	\label{alg:buildtree2}
	\begin{algorithmic}[1]
		\STATE{Generate a family of i.i.d. samples $\{\omega\}_{s=1}^L$. For each $\omega_s$, construct a body-fitted and quasi-uniform partition of the cell
		 ${\cal Q}^{\mathbf k}_M$ according to $A(y,\omega_s)$ for $y \in {\cal Q}^{\mathbf k}_M$.}
		\STATE{For  $\mathbf k \in \mathbb{Z}$ such that $\varepsilon{\cal Q}_M^\mathbf{k}\cap  {D}\neq \emptyset$,
			compute a set of finite element cell solutions
			$\mathbb{N}_{e_i}^{\mathbf{k},h}(y,\omega_s)$ by solving cell problem \eqref{eq:new-2.23} and the equivalent matrix
			$\hat A^{\mathbf{k},h}(\omega_s)=\big(\hat a_{ij}^{\mathbf{k},h}(\omega_s)\big)$ by \eqref{eq:homo-algo-1}.}
		\STATE{For each $\omega_s$, solve the finite element solution $\hat u^{h_1}(x,\omega_s)$ from \eqref{eq:new-2.29-1}.}
		\STATE{Calculate the reference solution for the finite element two-stage stochastic homogenization method by
			\begin{equation*}
			 \hat u^{L,h_1}(x)=\frac1 L{\sum_{s=1}^{L}\hat u^{h_1}(x,\omega_s)}.
			\end{equation*}
		}
	\end{algorithmic}
\end{algorithm}

\section{Numerical experiments} \label{sec:experiments}
In this section, we present several 2D numerical experiments to evaluate the
performance of the proposed two-stage stochastic homogenization method for the random diffusion equation \cref{eq:problem}. We mainly focus on the verification of the accuracy and efficiency of the finite element approximation of the two-stage method. Our numerical experiments are performed on a desktop workstation with 16G memory and  a 3.4GHz Core i7 CPU.
We set $f(x)= 10$ in all the tests. The cell ${\cal Q}_M$ in the two-stage stochastic homogenization method is taken as $Q=(0,1)^2$ with $M=1$.
The computational domain ${D}=(0,1)^2$ and $\varepsilon=1/8$.

\subsection{Convergence of the equivalent matrix}
In this test, we study the convergence of the numerical empirical
mean $\mu_L$ given in \eqref{eq:new-2.25},
which is used in the two-stage stochastic homogenization method.
Assume ${\cal Q}_M=Q=Q_1\cup Q_2$ with $Q_1\cap Q_2=\emptyset$
and $Q_2=(0.25,0.75)^2$. The random coefficients in the  diffusion equation \cref{eq:problem} are chosen as
\begin{equation}\label{eq:example-1}
a_{ij}(\frac{x}{\varepsilon},\omega)=\left\{
\begin{split}
& 3\delta_{ij}+(1+\sin(2\pi \frac{x_1}{\varepsilon})\sin(2\pi \frac{x_2}{\varepsilon})\delta_{ij})Z_{\mathbf k}(\omega)\quad x \in \varepsilon{(Q_1+\mathbf k)},\\
& 300\delta_{ij}+(50+\sin(2\pi \frac{x_1}{\varepsilon})\sin(2\pi \frac{x_2}{\varepsilon})\delta_{ij})Z_{\mathbf k}(\omega)\quad x \in  \varepsilon{(Q_2+\mathbf k)}.
\end{split}\right.
\end{equation}
where the i.i.d. random variables $\big(Z_{\mathbf k}(\omega)\big)_{\mathbf k\in\mathbb {Z}^2}$ satisfy a truncated normal distribution.
The probability density function for the truncated normal distribution is given by
\begin{equation}\label{eq:example-1-1}
f(\omega,-b,b)=\left\{
\begin{split}&\frac{\phi(\omega)}{\Phi(b)-\Phi(-b)}\quad &\omega\in[-b,b],\\
&0\quad &\hbox{otherwise}.\end{split}\right.
\end{equation}
where $\phi(\omega)=\frac{1}{\sqrt{2\pi}}\exp{\left(-\frac 12 \omega^2\right)}$, $\Phi(\omega)=\frac 1 2(1+\hbox{erf}(\omega/\sqrt{2}))$,
and $b=1.5$.

For the comparison purpose, we calculate $A^\star_N=(\hat a_{ij,N}^\star(\omega))$ defined by the ``cut-off" approach in \eqref{eq:approx-homo-coeff}
and take it as the reference solution. Since $A^\star_N$ converges to $A^\star$ as $N\rightarrow \infty$,
we compare the values of $\mu_L$ and $A^\star_N$ by setting $L=N^2$ to study the convergence behavior of $\mu_L$.
We generate a family of random variables $ Z_{\mathbf k}(\omega)$ with $1\leq k_1\leq N$ and $1\leq k_2 \leq N$.
In the two-stage stochastic homogenization method, the $L$ samples are then taken as
$ Z_{0}(\omega_s)=Z_{\mathbf k}(\omega_s)$ with $s=k_2 N+k_1$.
To neglect the numerical errors coming from the finite element discretization,
we use a finer quasi-uniform mesh of $\mathbb{T}^{\mathbf{k}}_h$ with $h=1/60$.

The numerical results for $\mu_L$ and $A^\star_N$
with $L=N^2=4,16,\cdots,1024$ are given in  Table \ref{tab-e0-1}. From the table we
observe that the relative errors for the equivalent matrix  $\mu_L$ by taking the stochastic homogenization matrix
$A^\star_N$ as reference increase as $N$ increases and stop at a small value (about 4\%).
Since $A^\star_N$ converges to $A^\star$ as $N\rightarrow \infty$, thus one can take $\mu_L$
as a valid approximation of $A^\star$. To study statistic fluctuations of the equivalent matrix  $\mu_L$ and the stochastic homogenization matrix
$A^\star_N$, we take $N=22$ as an example and run the simulation with  twenty sets of  i.i.d. samples. Each set consists of 484 samples.
The numerical expectation and variance are given in Table \ref{tab-e0-2}, which clearly show that  the equivalent matrix  $\mu_L$
is a good approximation for the stochastic homogenization matrix  $A^\star$. The total computation time for each set of the two approaches
is also reported in Table \ref{tab-e0-2}, which demonstrates that the proposed two-stage stochastic  homogenization method is more efficient than
the classical stochastic homogenization method with almost the same accuracy. As a comparison to the cell problem posed on $Q_{N}$ for $N=22$ in the ``periodization" stochastic homogenization method,
the proposed two-stage stochastic homogenization method only needs to deal with $N^2$ cell problems defined on the unit cell $Q$,
which is the main reason for the computational saving and the efficiency improvement.  It should be pointed out that the  $N^2$ cell problems can be solved naturally in parallel.

\begin{table}[!htbp]
	\caption{\label{tab-e0-1}The comparison of equivalent coefficients
		$\mu_{11,L}^0$ and 	stochastic homogenization coefficients $\hat{a}^\star_{11,N}$ for  cell $Q_N$.
		The ``Error" in the table is defined as $|\mu_{11,L}^0-\hat{a}^\star_{11,N}|/\hat{a}^\star_{11,N}$. The comparisons for the other three entries of the matrix are similar but not showed here.}
	\footnotesize
	\begin{center}
		\begin{tabular}{cccccccc}
			\toprule
			$L=N^2$ & $\mu_{11,L}^0$ & $\hat{a}^\star_{11,N}$ & Error & $L=N^2$ &  $\mu_{11,L}^0$ & $\hat{a}^\star_{11,N}$  & Error \\
			\midrule
			4    & 4.9437 & 4.8932 & 0.0103 &  324    & 5.1923 & 5.0009 & 0.0383 \\
			16   & 5.0970 &  4.8020 & 0.0614  &  400    & 5.2592 & 4.9960 & 0.0527 \\
			36   &  5.3440 &  5.2024 & 0.0272  &  484    & 5.2100 &4.9941 & 0.0432 \\
			64   &  5.5265 &  5.1084 & 0.0818  &  576   & 5.3724 & 5.0974 & 0.0539 \\
			100   &  5.3234 &  5.0461 &  0.0550 &  676    & 5.2123 & 4.9562 & 0.0517 \\
			144   &  5.1703 &  4.8894 & 0.0575  &  784    & 5.2520 & 5.0438 & 0.0413 \\
			196   &  5.1830 &  4.8505 &  0.0685 &  900    & 5.2447 & 5.0212 & 0.0445 \\
			256   &  5.1892 &  4.8886 &0.0615   &  1024    & 5.1899 & 4.9759 & 0.0430 \\
			\bottomrule
		\end{tabular}
	\end{center}
\end{table}

\begin{table}[!htbp]
	\caption{\label{tab-e0-2}The comparison of the numerical expectation and variance of equivalent matrix and stochastic homogenization matrix.}
	\footnotesize
	\begin{center}
		\begin{tabular}{cccc}
			\toprule
			&  Expectation & Variance  & Compute time (s)  \\
			\midrule
			$\mu_{11,L}^0$  & 5.2130 & 0.1320 & 125.9\\
			$\hat{a}^\star_{11,N}$                                     &  4.9991 &  0.1334&  917.5\\
			
			\bottomrule
		\end{tabular}
	\end{center}
\end{table}

\subsection{Verification of accuracy and efficiency}
In order to validate the accuracy and efficiency of the proposed two-stage stochastic
homogenization method, three kinds of tests (Test A, B, C)
are performed in this subsection. In Test A, the composite materials
have periodic structure and random coefficients;
in Test B,   the materials have random structure and deterministic coefficients;
in Test C,   both the structure and coefficients of the materials are random.

\paragraph*{Test A: Composite materials with periodic structure and random coefficients}
In this simulation, we consider two types (Type I and Type II) of composite
materials with periodic structure and random parameters. The computational domain
${D}$ is decomposed into $8\times8$ cells and all cells have the same geometry constructed by matrix (denoted by $Q_1$) and inclusions (denoted by $Q_2$).
The unit cell of Type I  includes a square inclusion as shown in Figure  \ref{fig:example1-1}-(a).  Figure \ref{fig:example1-1}-(b) shows the unit cell of Type II  which contains 70 elliptical inclusions with uniform random distribution,
which is generated by the take-and-place algorithm \cite{wang1999mesoscopic}. The random coefficients for matrix and inclusion are given by \cref{eq:example-1},
in which  i.i.d. random variables $Z_{\mathbf k}(\omega)$  satisfy the uniform distribution over $[-1,1]$ or
a  truncated normal distribution with the probability density function defined by \cref{eq:example-1-1}.

\begin{figure}[h]
	\begin{center}
		{\tiny(a)} \includegraphics[width=0.44\textwidth]{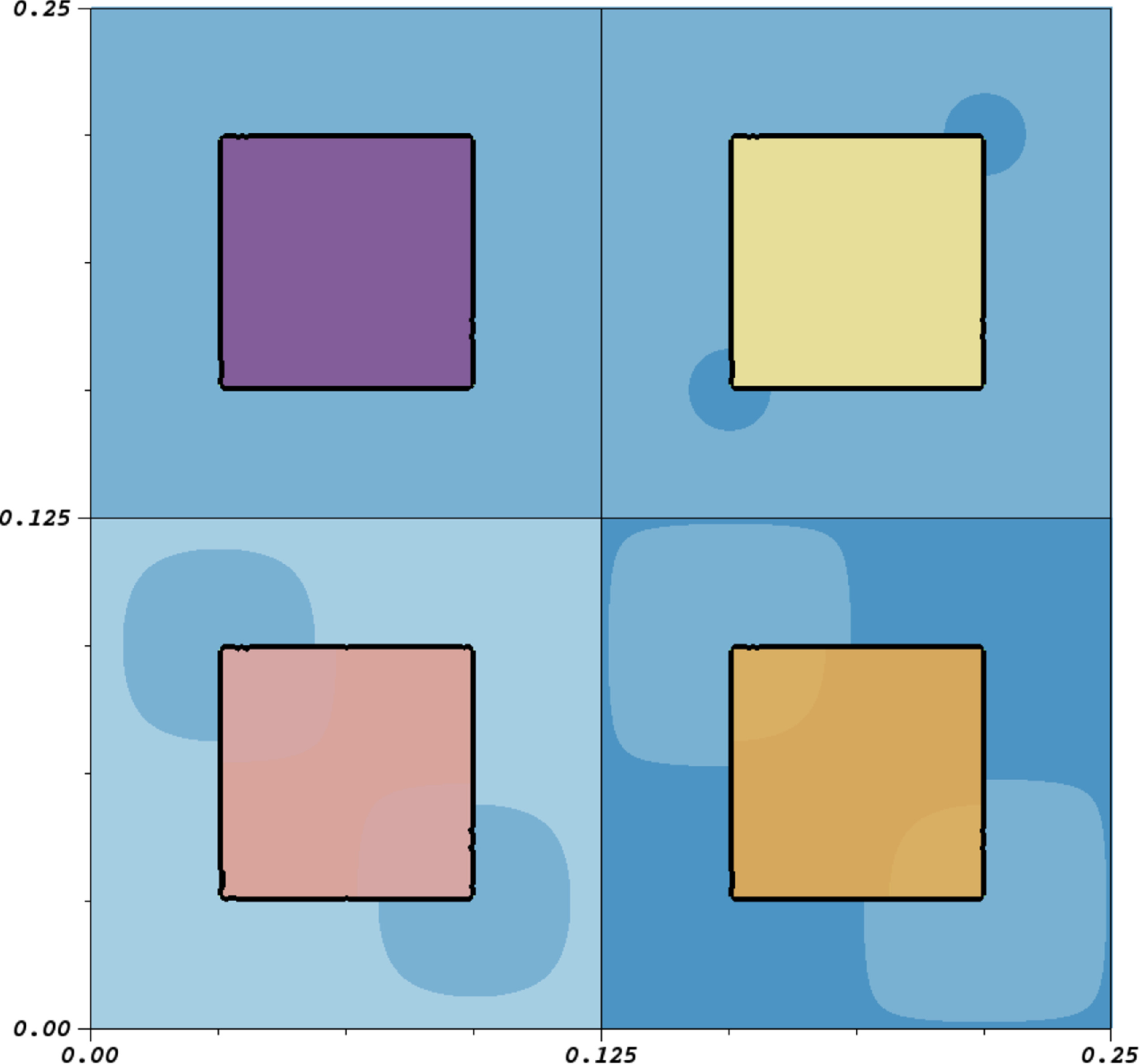}
		{\tiny(b)}    \includegraphics[width=0.44\textwidth]{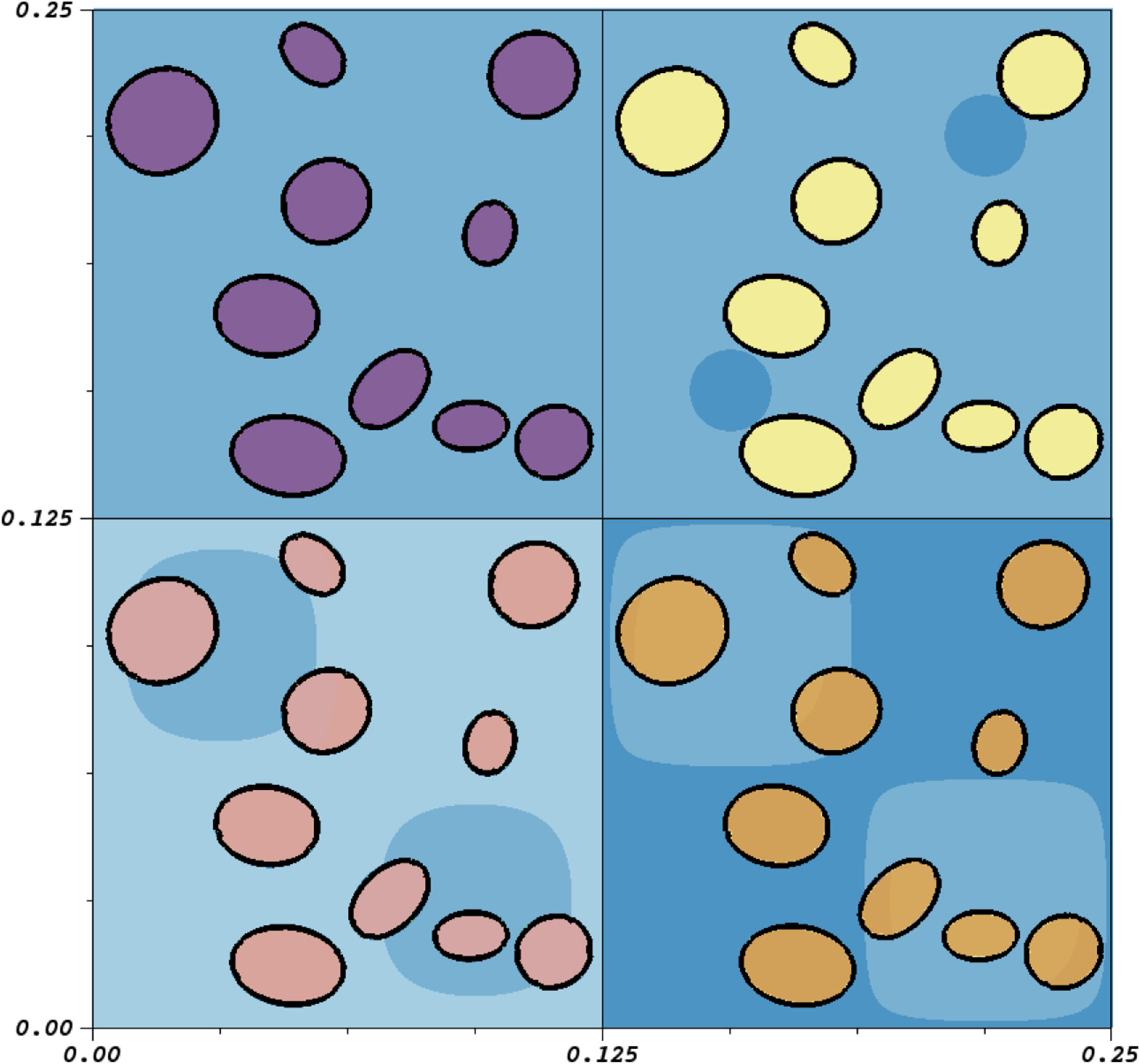}
	\end{center}
	\caption{Test A: Composite materials with periodic structure and random coefficients.
		Pseudo-color of $a_{11}(x,\omega)$ on the domain $(0,\ 0.25)^2$ is plotted. (a) Type I: a square inclusion. (b)
		Type II: 10 elliptical inclusions with deterministic major/minor axis and rotation angles in each cell are taken as an example. }\label{fig:example1-1}
\end{figure}

\begin{figure}[h]
	\begin{center}{\tiny(a)}   \includegraphics[width=0.44\textwidth]{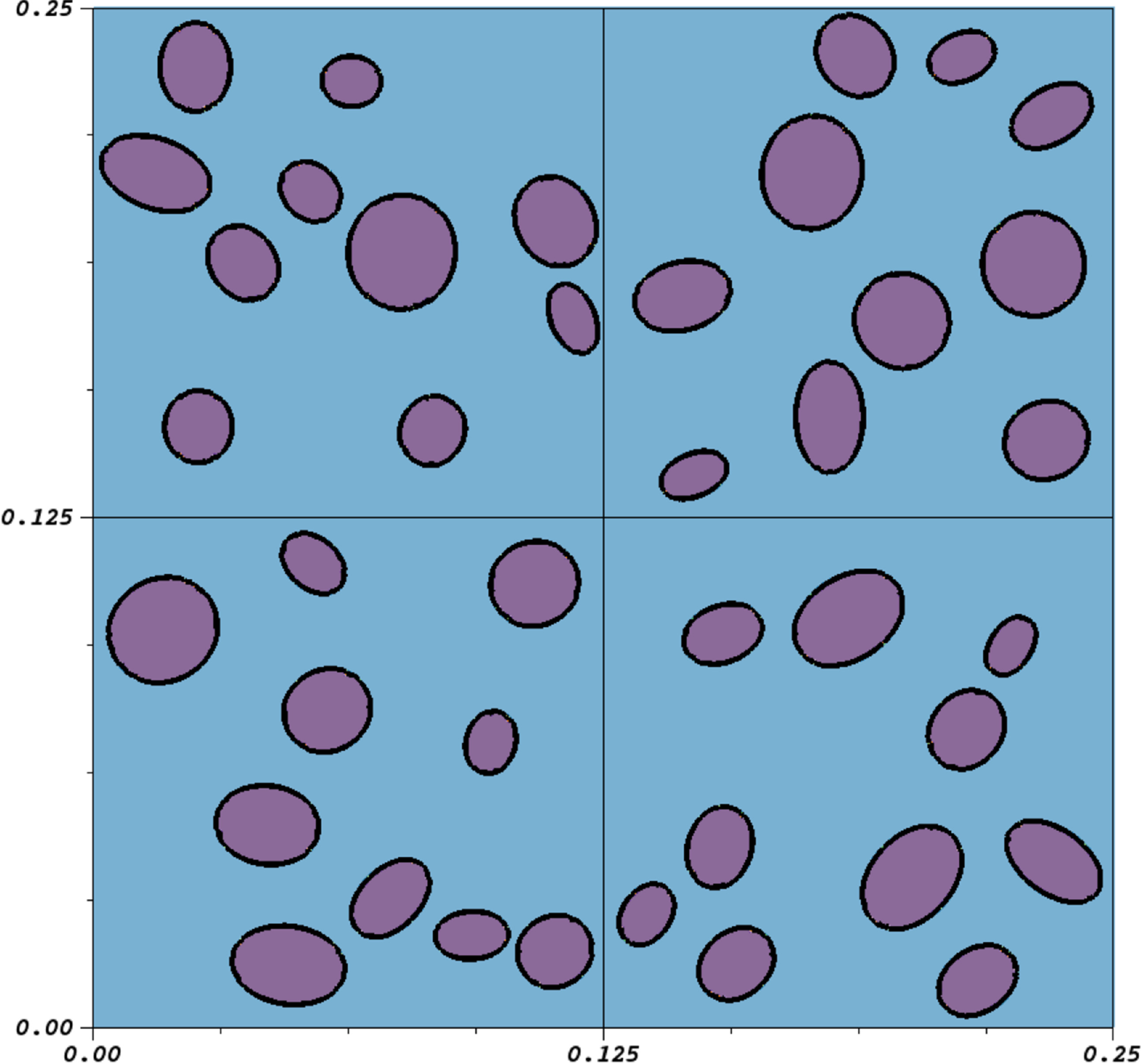}
		{\tiny(b)}  \includegraphics[width=0.44\textwidth]{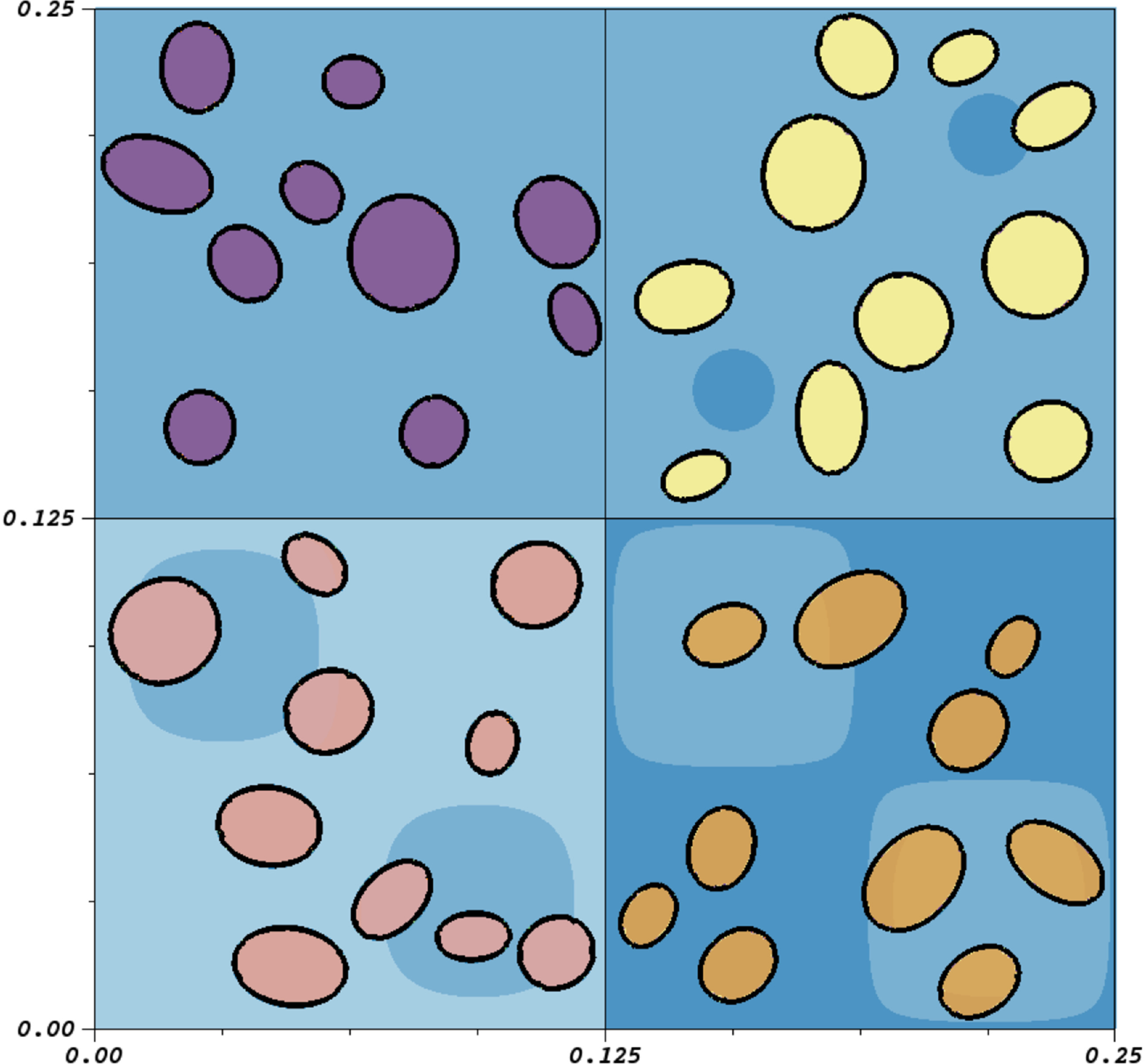}
	\end{center}
	\caption{Pseudo-color of $a_{11}(x,\omega)$ on the domain $(0, 0.25)^2$ is plotted.
		Elliptical inclusions with random position, major/minor axis, and rotation angles in each cell.
		(a) Test B: Composite materials with  random geometry structure and deterministic coefficients.
		(b)  Test C: Composite materials with random geometry structure and random coefficients.}\label{fig:example1-2}
\end{figure}

\begin{figure}[h]
	\centering
	{\tiny(a)}\includegraphics[width=0.45\textwidth]{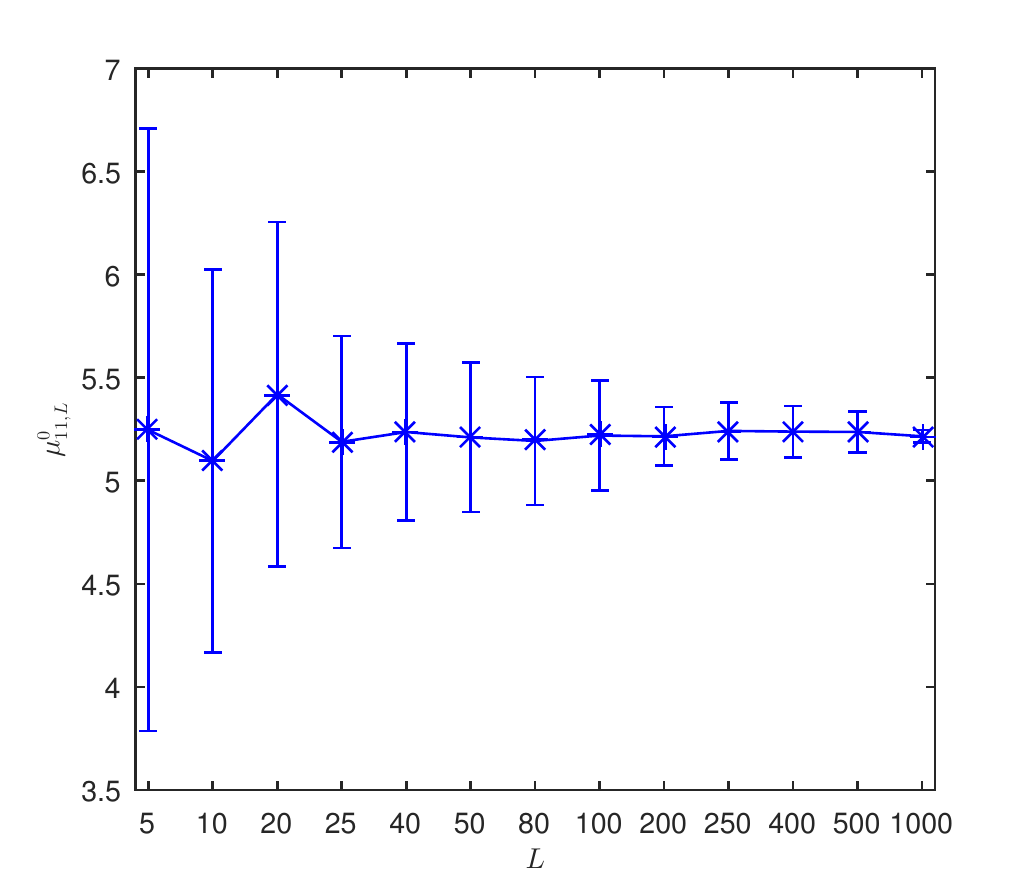}~
	{\tiny(b)}\includegraphics[width=0.45\textwidth]{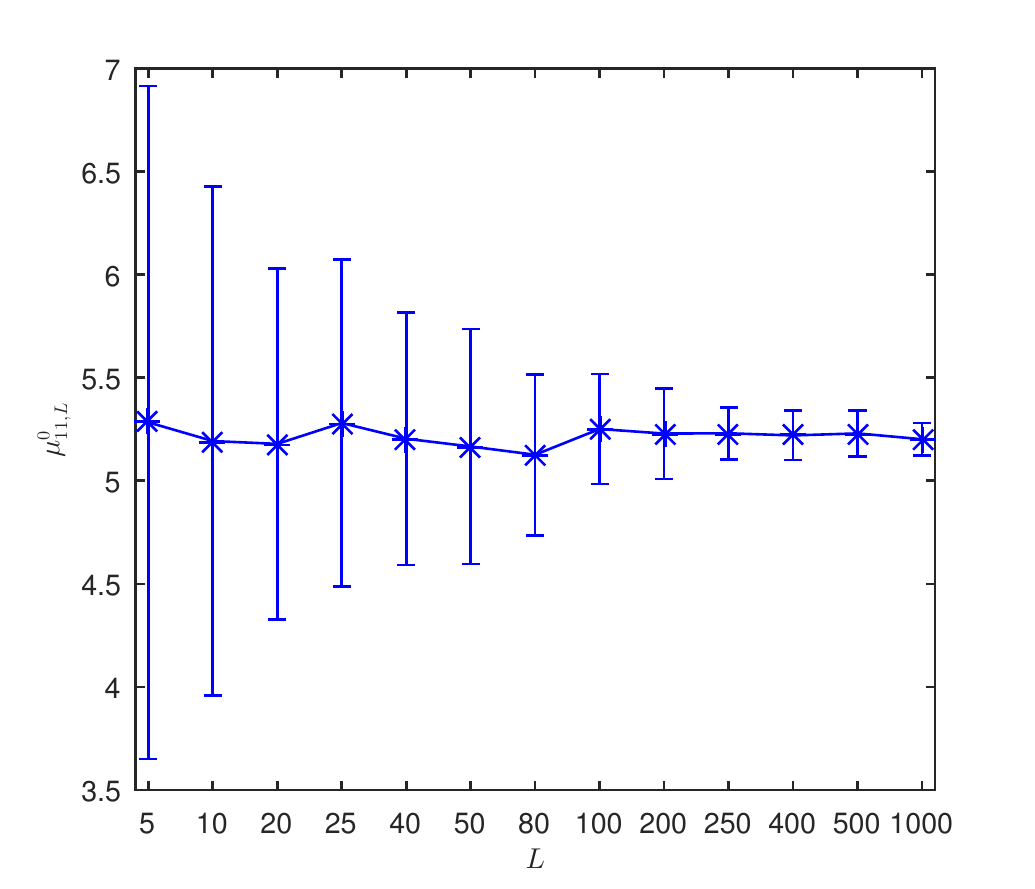}
	{\tiny(c)}\includegraphics[width=0.45\textwidth]{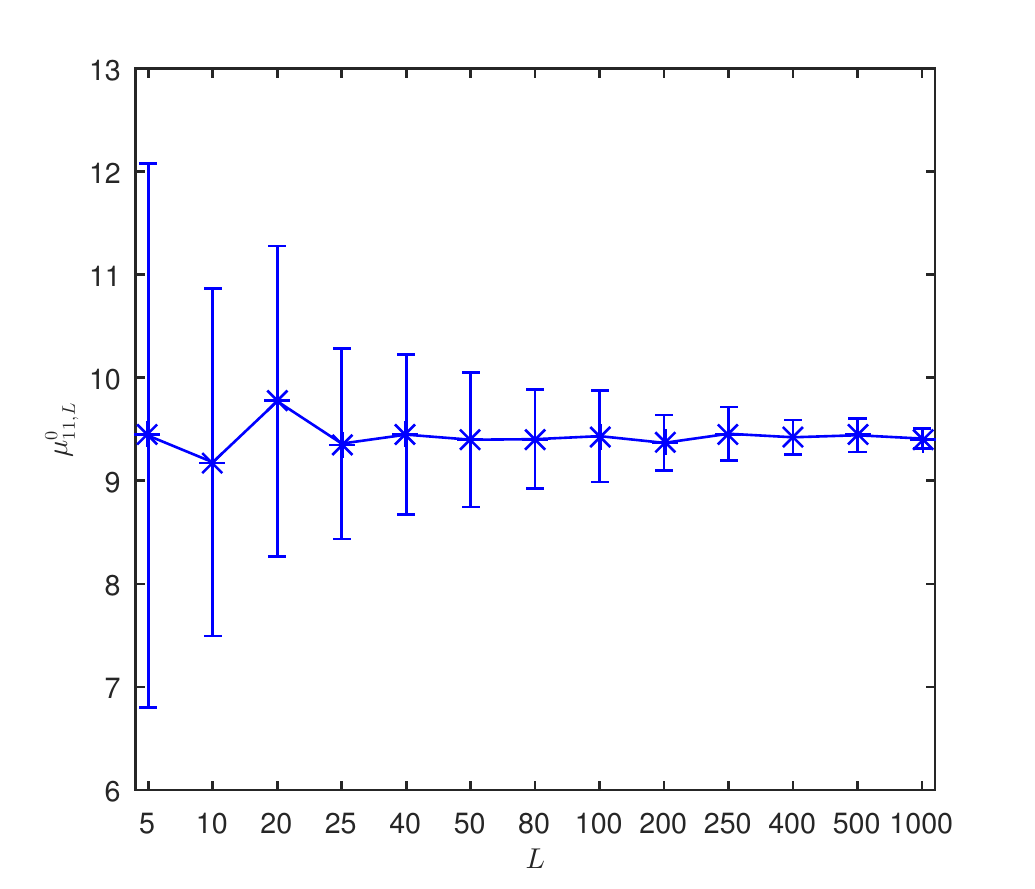}~
	{\tiny(d)}\includegraphics[width=0.45\textwidth]{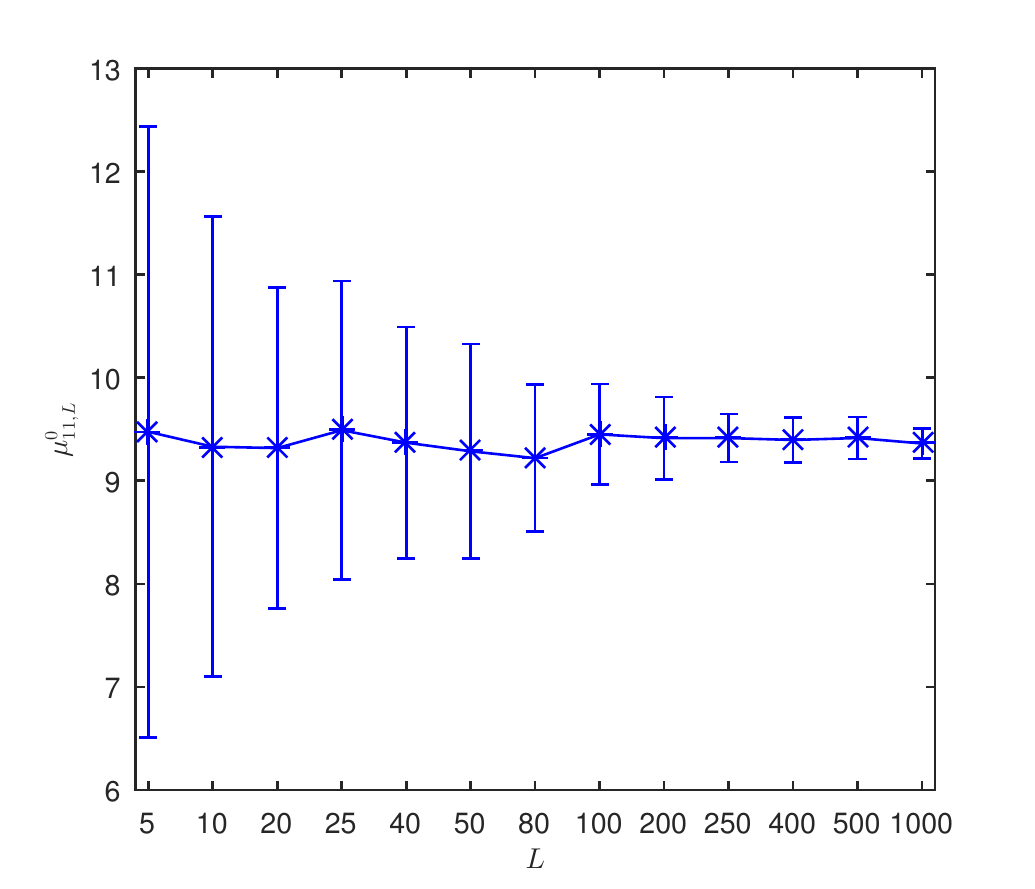}
	\caption{
		The numerical expectation and fluctuation of equivalent coefficient $\mu^0_{11,L}$.
		(a) Type I of Test A with coefficients satisfying uniform distribution over $[-1,1]$;
		(b) Type I of Test A with coefficients satisfying  the truncated normal distribution;
		(c) Type II of Test A with coefficients satisfying uniform distribution over $[-1,1]$;
		(d) Type II of Test A with coefficients satisfying  the truncated normal distribution.
	}\label{fig:example1-3}
\end{figure}

\begin{figure}[h]
	\centering
	{\tiny(a)}\includegraphics[width=0.45\textwidth]{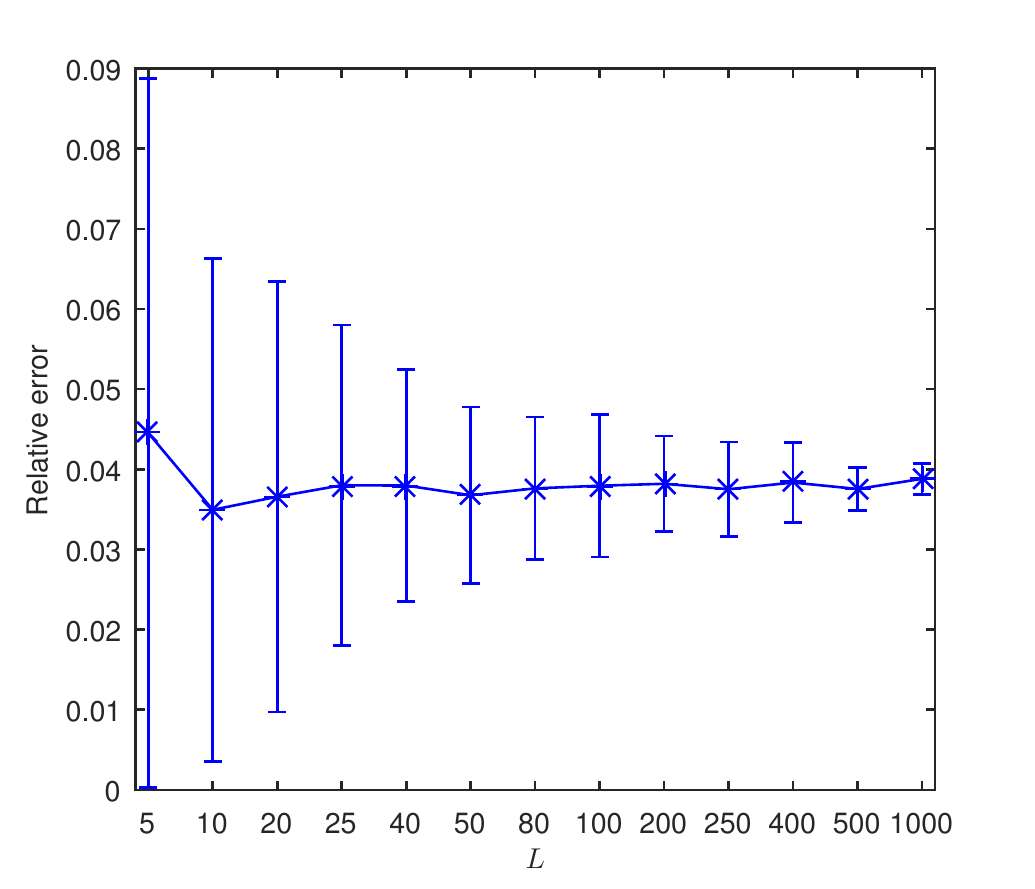}~
	{\tiny(b)}\includegraphics[width=0.45\textwidth]{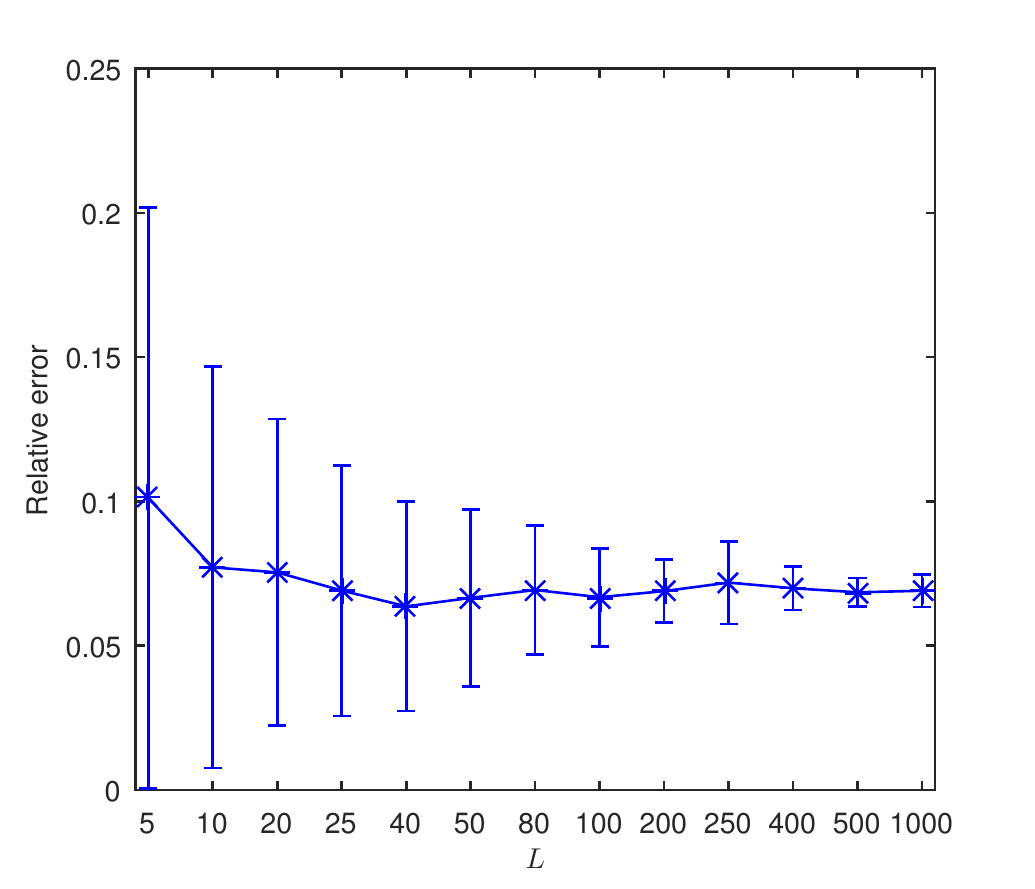}
	{\tiny(c)}\includegraphics[width=0.45\textwidth]{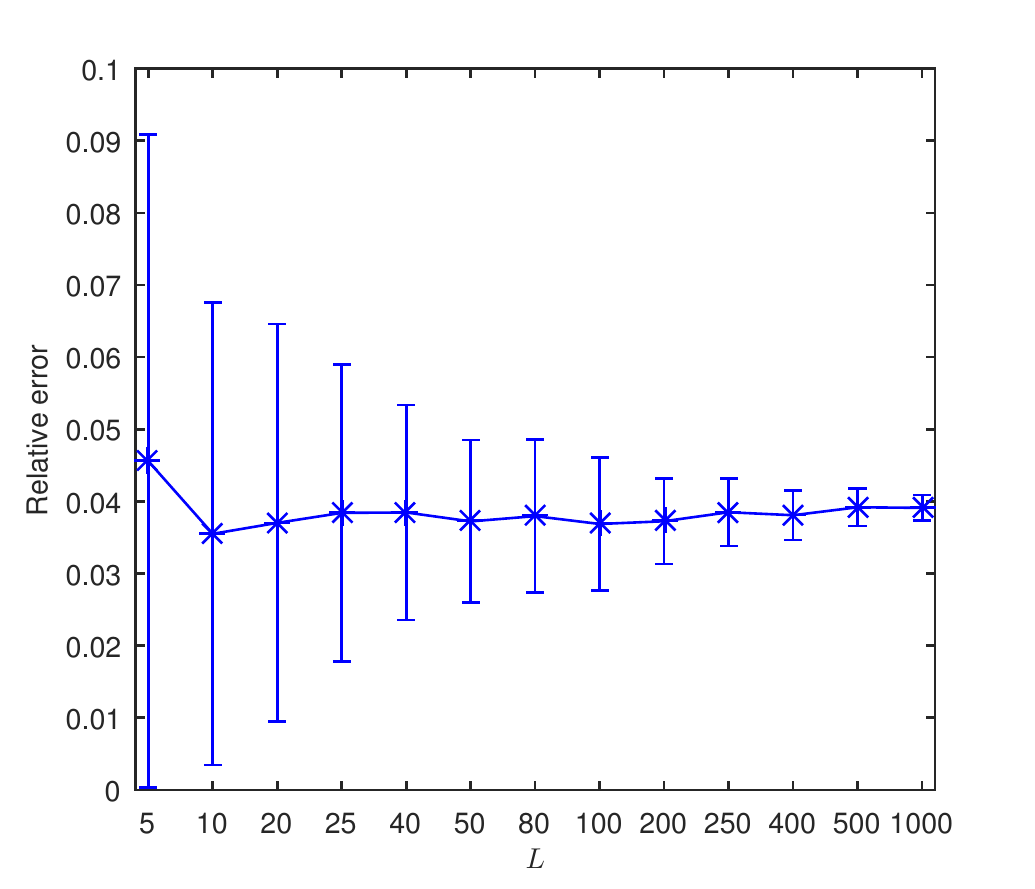}~
	{\tiny(d)}\includegraphics[width=0.45\textwidth]{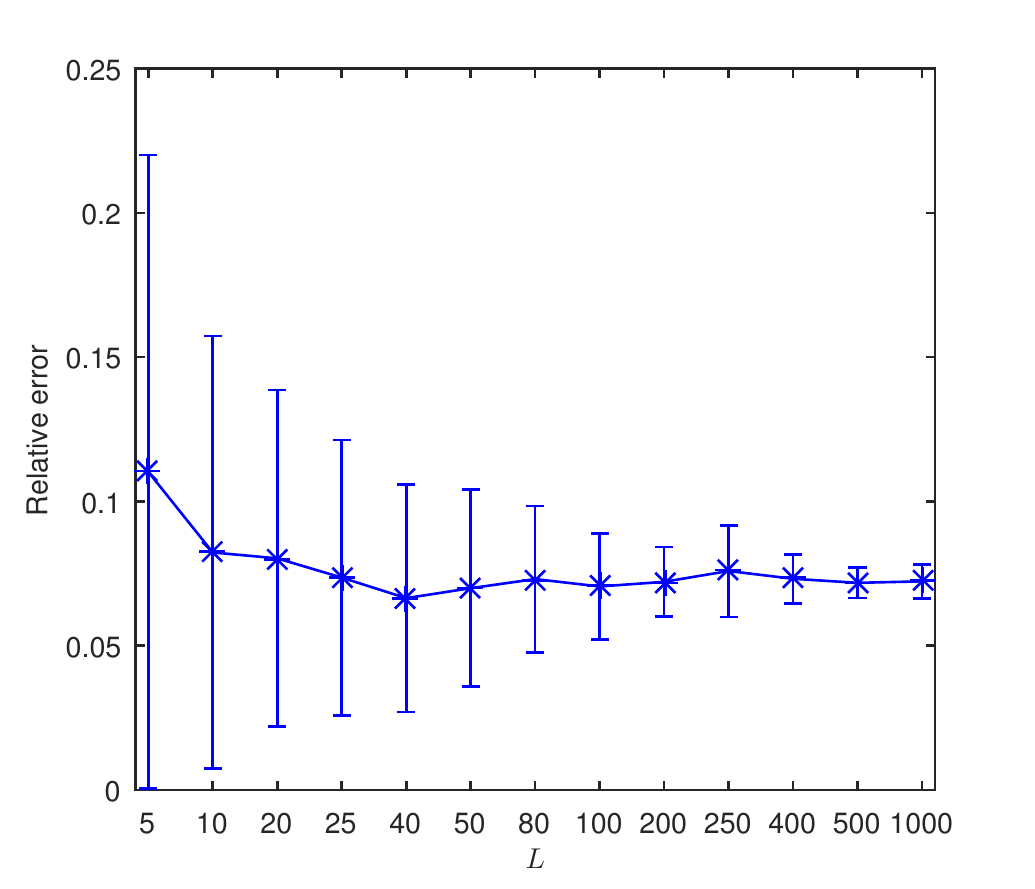}
	\caption{The relative errors of the two-stage stochastic homogenization method.
		(a) Type I of Test A with coefficients satisfying uniform distribution over $[-1,1]$;
		(b) Type I of Test A with coefficients satisfying  the truncated normal distribution;
		(c) Type II of Test A with coefficients satisfying uniform distribution over $[-1,1]$;
		(d) Type II of Test A with coefficients satisfying  the truncated normal distribution.
	}\label{fig:example1-4}
\end{figure}

\begin{figure}[h]
	\centering
	{\tiny(a)}\includegraphics[width=0.47\textwidth]{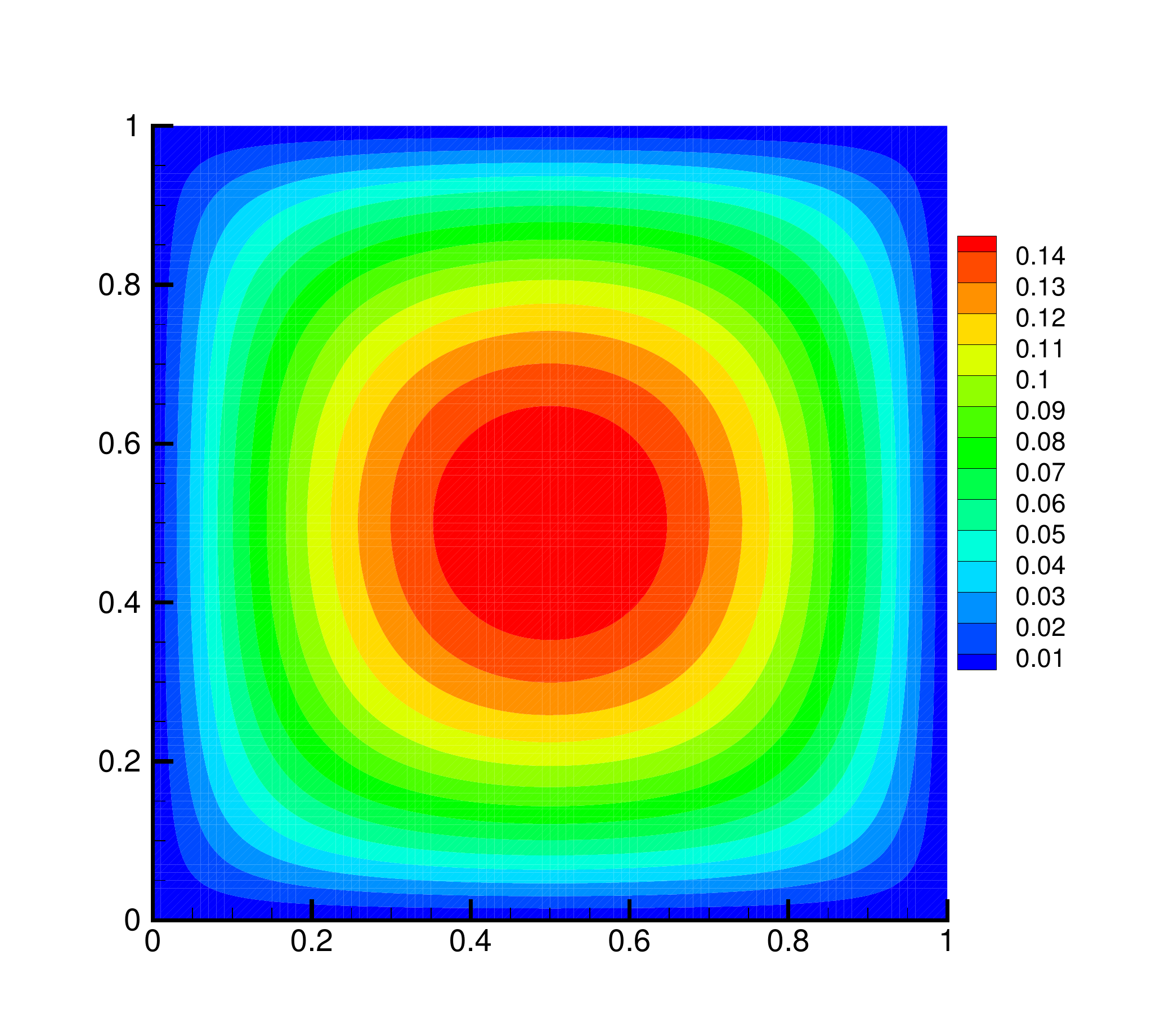}~
	{\tiny(b)}\includegraphics[width=0.47\textwidth]{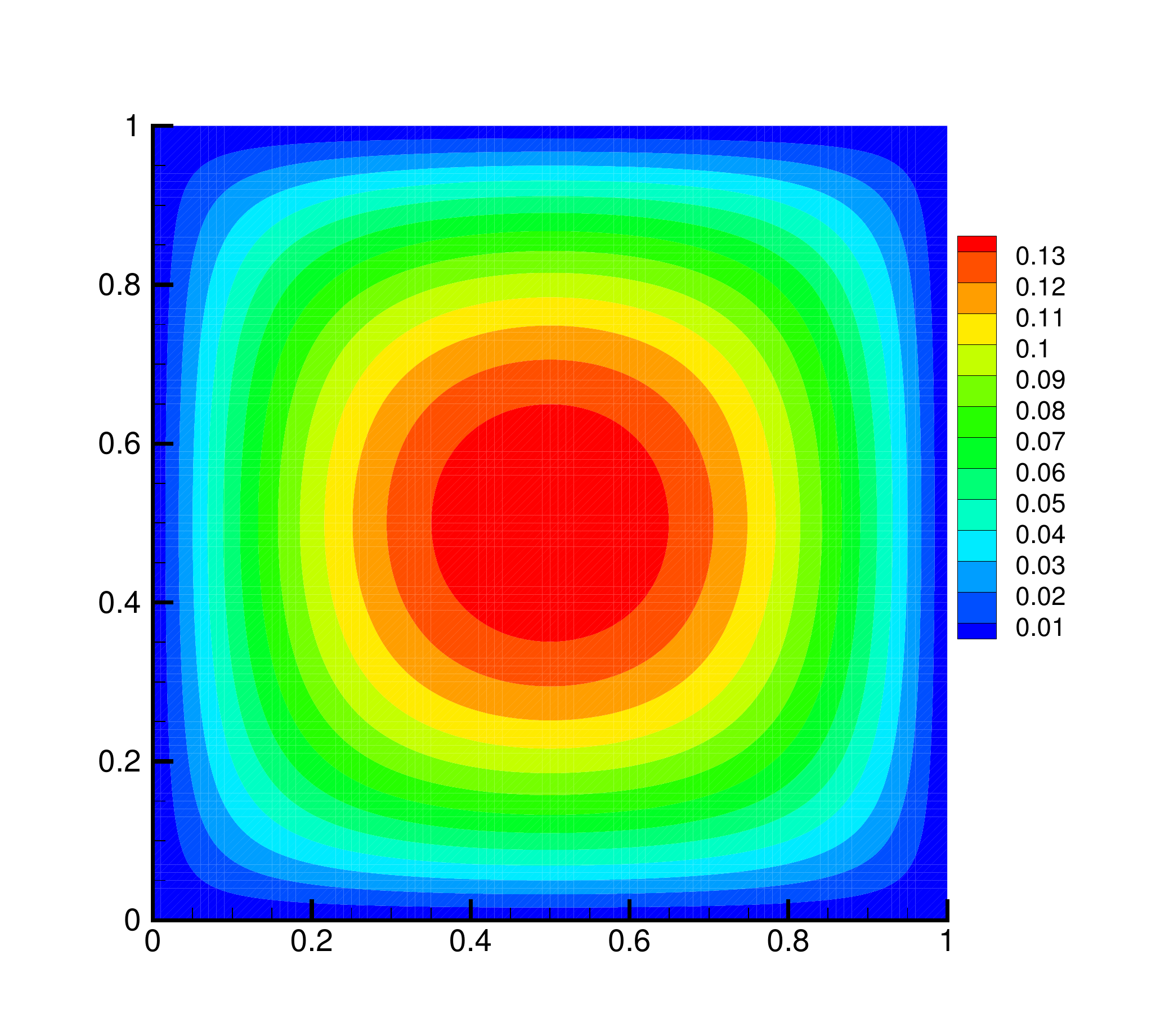}
	{\tiny(c)}\includegraphics[width=0.47\textwidth]{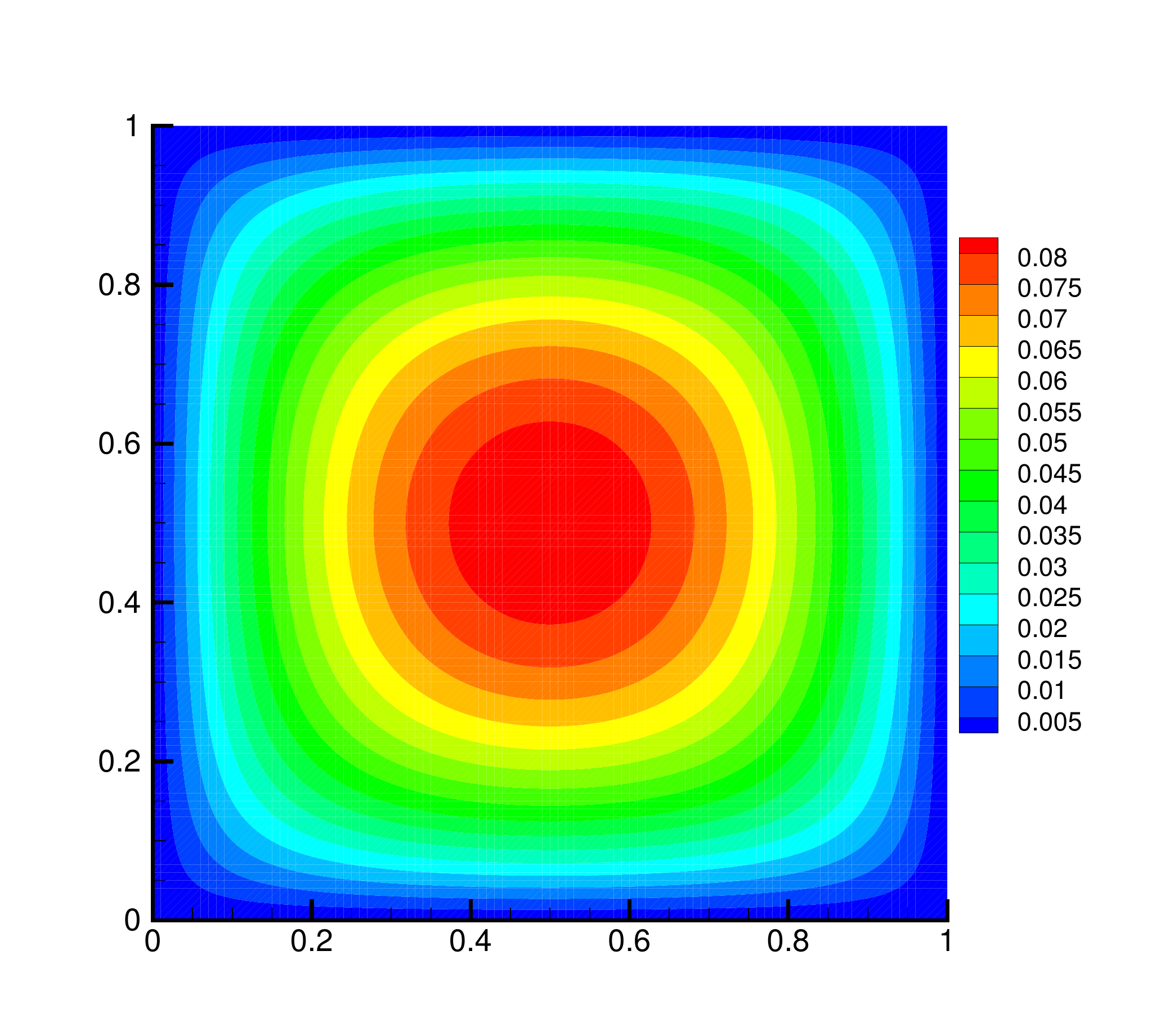}~
	{\tiny(d)}\includegraphics[width=0.47\textwidth]{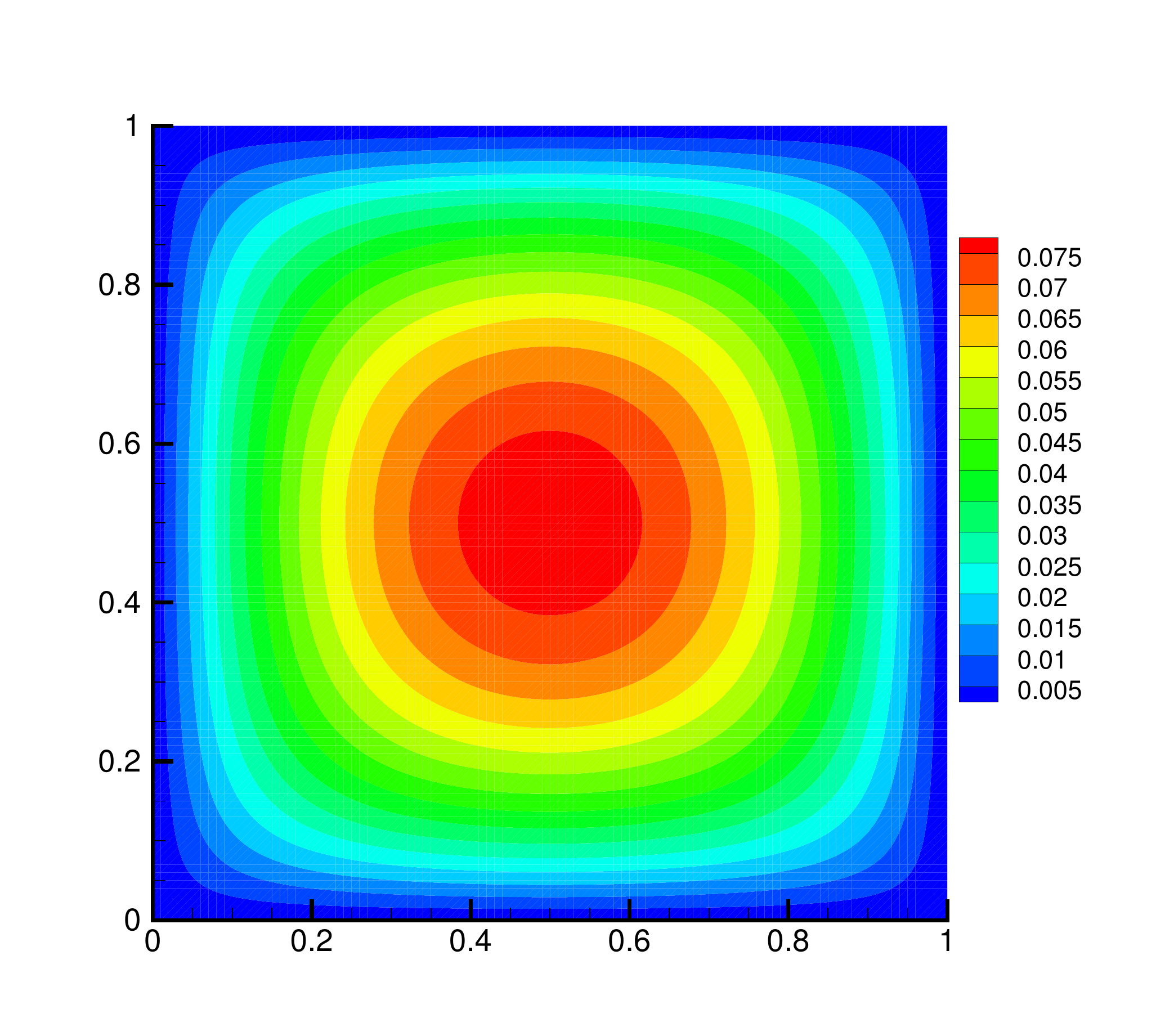}
	\caption{The contour plots of the solutions.
		(a) The two-stage stochastic homogenization solution $u_0^{0,h_0}$ for Type I of Test A with truncated normally distributed coefficients;
		(b) Reference solution $ \hat u^{L,h_1}  $ for Type I of Test A with standard normally distributed coefficients;
		(a) The two-stage stochastic homogenization solution $u_0^{0,h_0}$ for Type II of Test A with truncated normally distributed coefficients;
		(b) Reference solution $  \hat u^{L,h_1} $ for Type II of Test A with standard normally distributed coefficients.
	}\label{fig:example1-5}
\end{figure}

\paragraph*{Test B: Composite materials with  random structure and deterministic coefficients}
In this test, we consider composite  materials with random structure and deterministic coefficients.
As shown in Figure \ref{fig:example1-2}-(a), the computational domain $ {D}$ is decomposed into $8\times8$
cells and each cell contains 10 elliptical inclusions with uniform random distribution sample $\omega$, which is generated by the take-and-place algorithm \cite{wang1999mesoscopic}.
Denote $ {D}_1(\omega)$ as the matrix subdomain and ${D}_2(\omega)$ the inclusion subdomain. The deterministic coefficients in both subdomains are taken as
\begin{equation*}
a_{ij}(\frac{x}{\varepsilon},\omega)=\left\{
\begin{split}
&3\delta_{ij} &\quad x \in  {D}_1(\omega),\\
&300\delta_{ij} &\quad x \in  {D}_2(\omega).
\end{split}\right.
\end{equation*}

\begin{figure}[h]
	\centering
	{\tiny(a)}\includegraphics[width=0.45\textwidth]{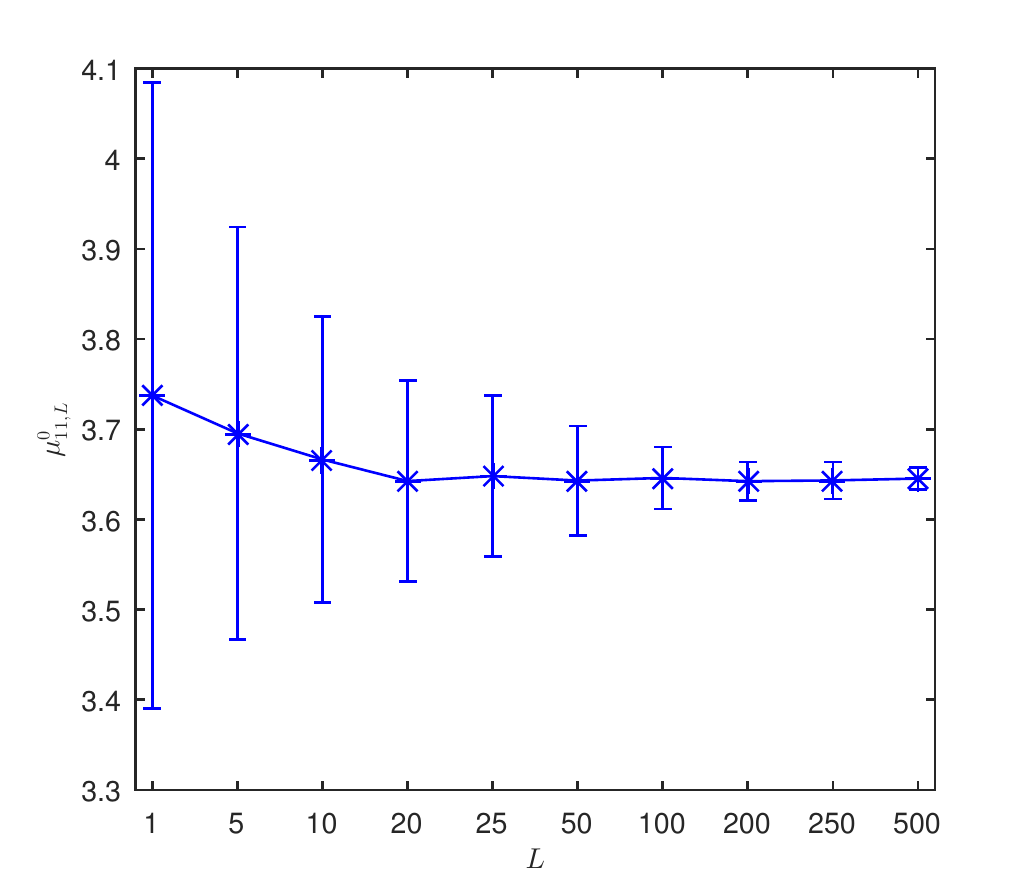}~
	{\tiny(b)}\includegraphics[width=0.45\textwidth]{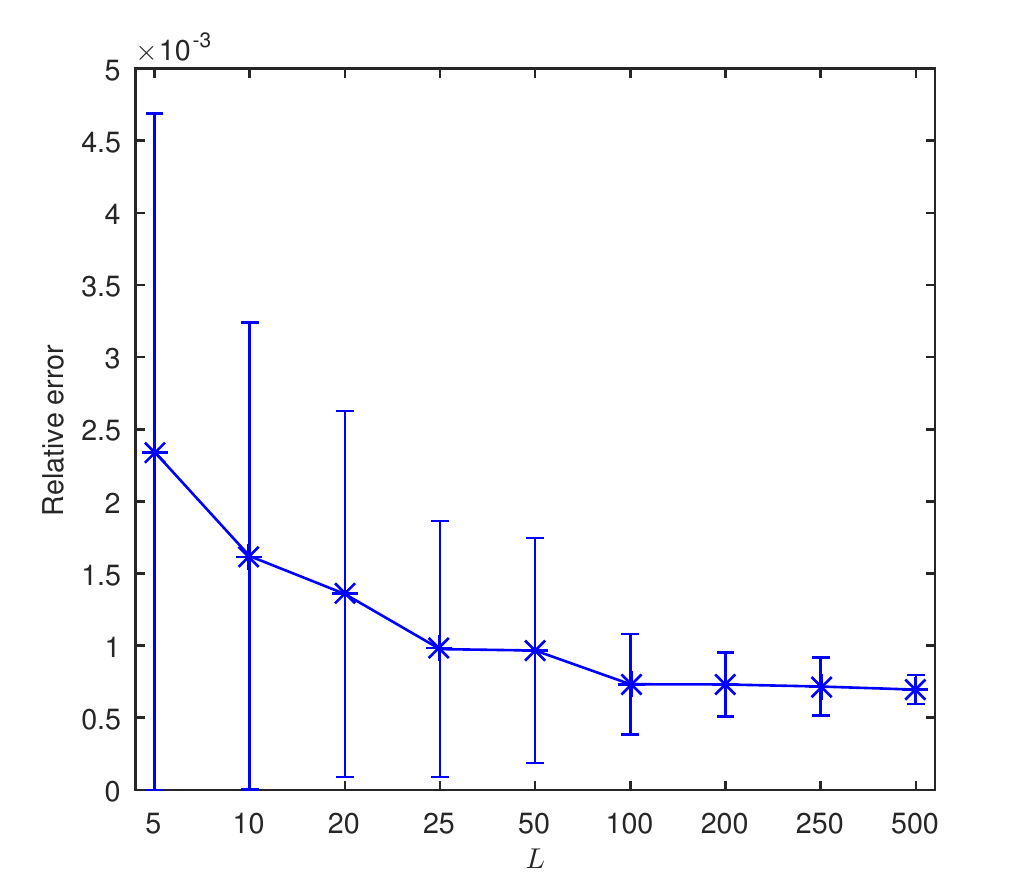}
	\caption{
		Test B: (a) the numerical expectation and  fluctuation of equivalent coefficient $\mu^0_{11,L}$;
		(b) the relative errors of the two-stage stochastic  homogenization method.
	}\label{fig:example2-1}
\end{figure}

\paragraph*{Test C: Composite materials with random structure and random coefficients}
This test can be viewed as a combination of Test A and B.
As shown in Figure \ref{fig:example1-2}-(b), the computational domain $ {D}$ is decomposed into $8\times8$ cells and each cell contains
10 elliptical inclusions with uniform random distribution sample $\omega_1$, which is generated by the take-and-place algorithm \cite{wang1999mesoscopic}.
Let ${D}_1(\omega_1)$ and $ {D}_2(\omega_1)$ have the same definitions as in Test B.
The random coefficients in both subdomains are now defined as
\begin{equation*}
a_{ij}(\frac{x}{\varepsilon},\omega)=\left\{
\begin{split}
&3\delta_{ij}+(1+\sin(2\pi \frac{x_1}{\varepsilon})\sin(2\pi \frac{x_2}{\varepsilon})\delta_{ij})Z_{\mathbf k}(\omega_2) & x \in  {D}_1(\omega_1)\cap\varepsilon{(Q+\mathbf k)},\\
&300\delta_{ij}+(50+\sin(2\pi \frac{x_1}{\varepsilon})\sin(2\pi \frac{x_2}{\varepsilon})\delta_{ij})Z_{\mathbf k}(\omega_2) & x \in  {D}_2(\omega_1)\cap\varepsilon{(Q+\mathbf k)}.
\end{split}\right.
\end{equation*}
where the i.i.d. random variables $\big(Z_{\mathbf k}(\omega_2)\big)_{\mathbf k\in\mathbb {Z}^2}$ satisfy the uniform distribution over $[-1,1]$
or  the truncated normal distribution with the probability density function defined by
\eqref{eq:example-1-1} .

\begin{table}[!htbp]
	\caption{\label{tab1}The number of degrees of freedom (Dof) used in the three tests.}
	\footnotesize
	\begin{center}
		\begin{tabular}{ccc}
			\toprule
			& Test A (Type I) & Test A (Type II), Test B,C  \\
			\midrule
			Cell problem & 3,600 & 14,400  \\
			Homogenized problem   &  10,000 &  10,000\\
			
			\bottomrule
		\end{tabular}
	\end{center}
\end{table}

\begin{figure}[h]
	\centering
	{\tiny(a)}\includegraphics[width=0.45\textwidth]{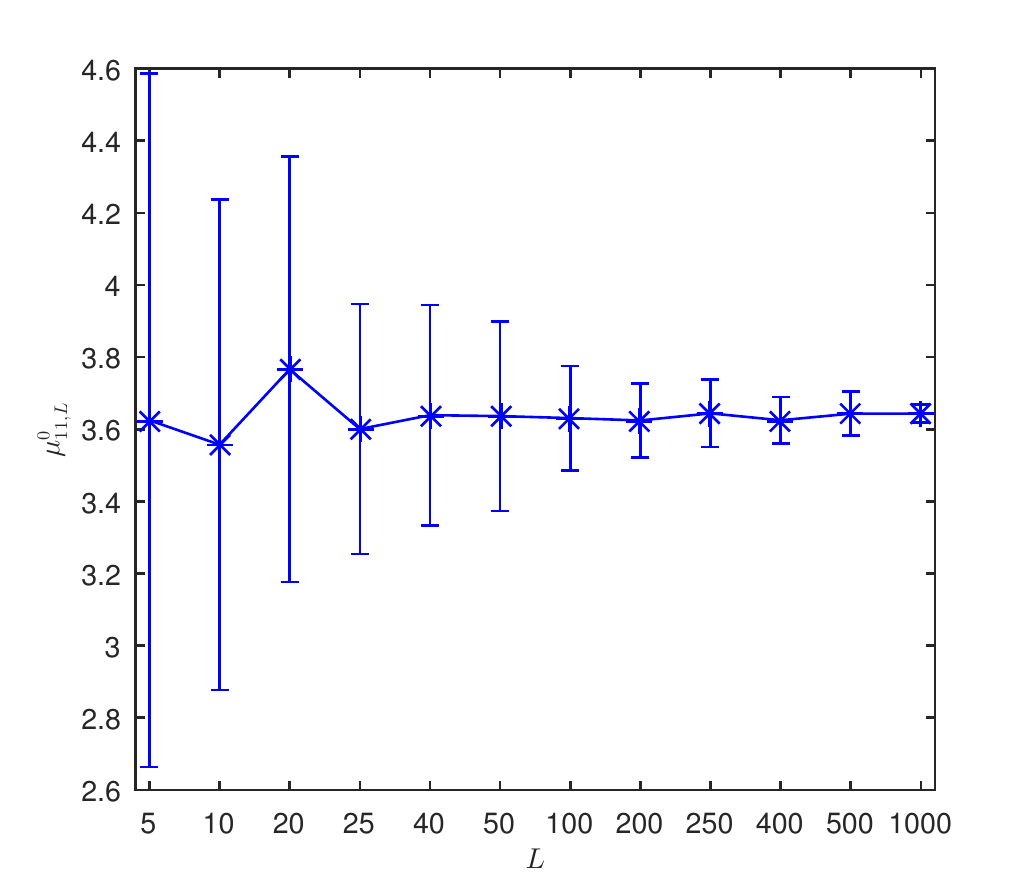}~
	{\tiny(b)}\includegraphics[width=0.45\textwidth]{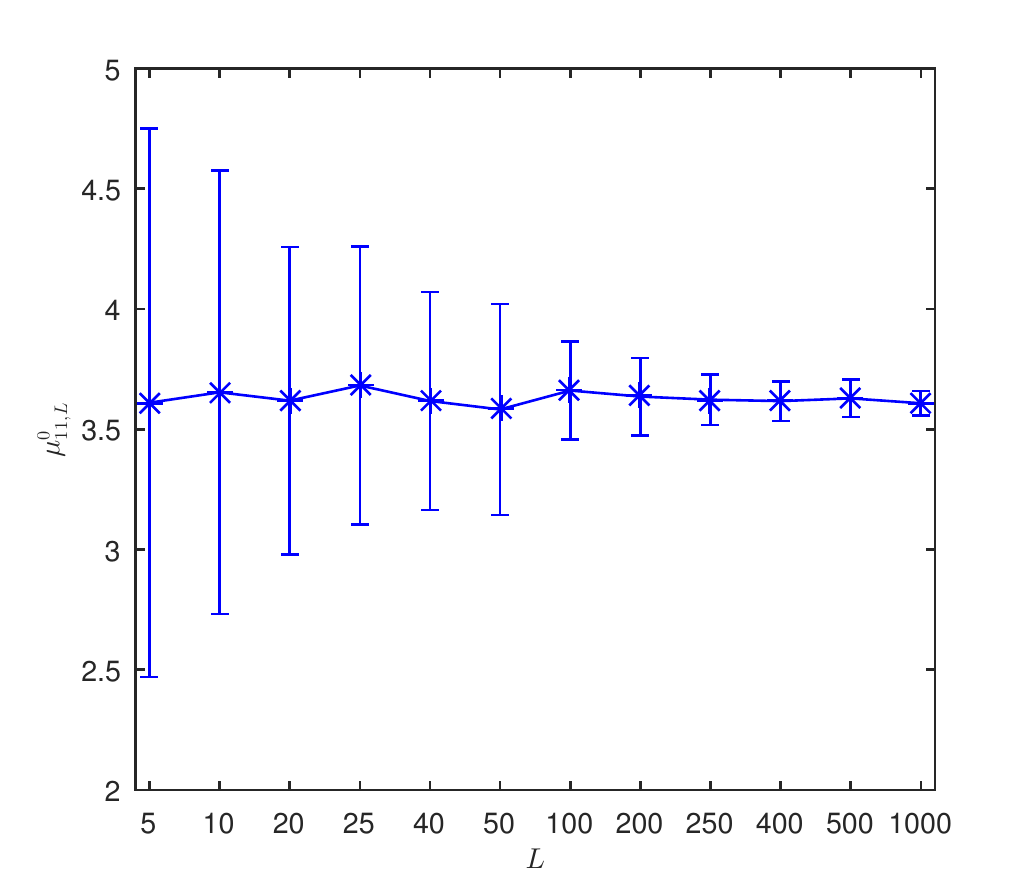}
	\caption{
		The numerical expectation and  fluctuation of equivalent coefficient $\mu^0_{11,L}$. (a)
		Test C with uniformly distributed coefficients;  (b) Test C with truncated normally distributed coefficients.
	}\label{fig:example3-1}
\end{figure}

\begin{figure}[h]
	\centering
	{\tiny(a)}\includegraphics[width=0.45\textwidth]{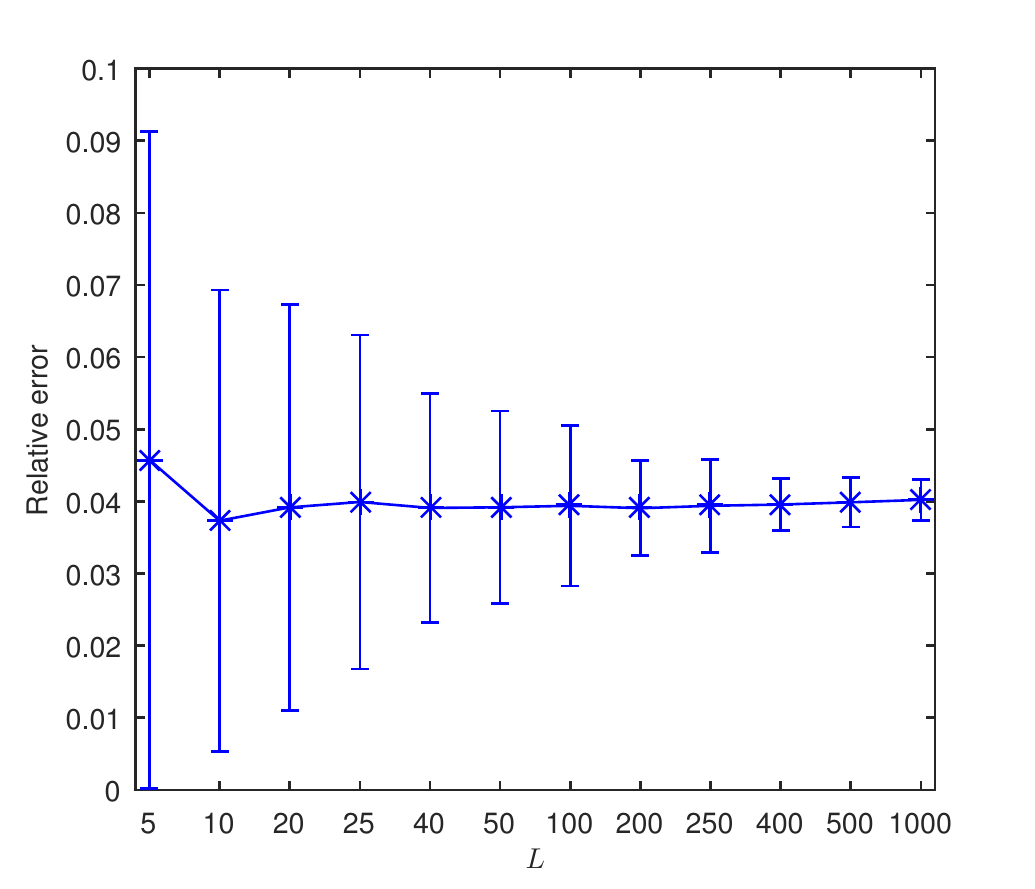}~
	{\tiny(b)}\includegraphics[width=0.45\textwidth]{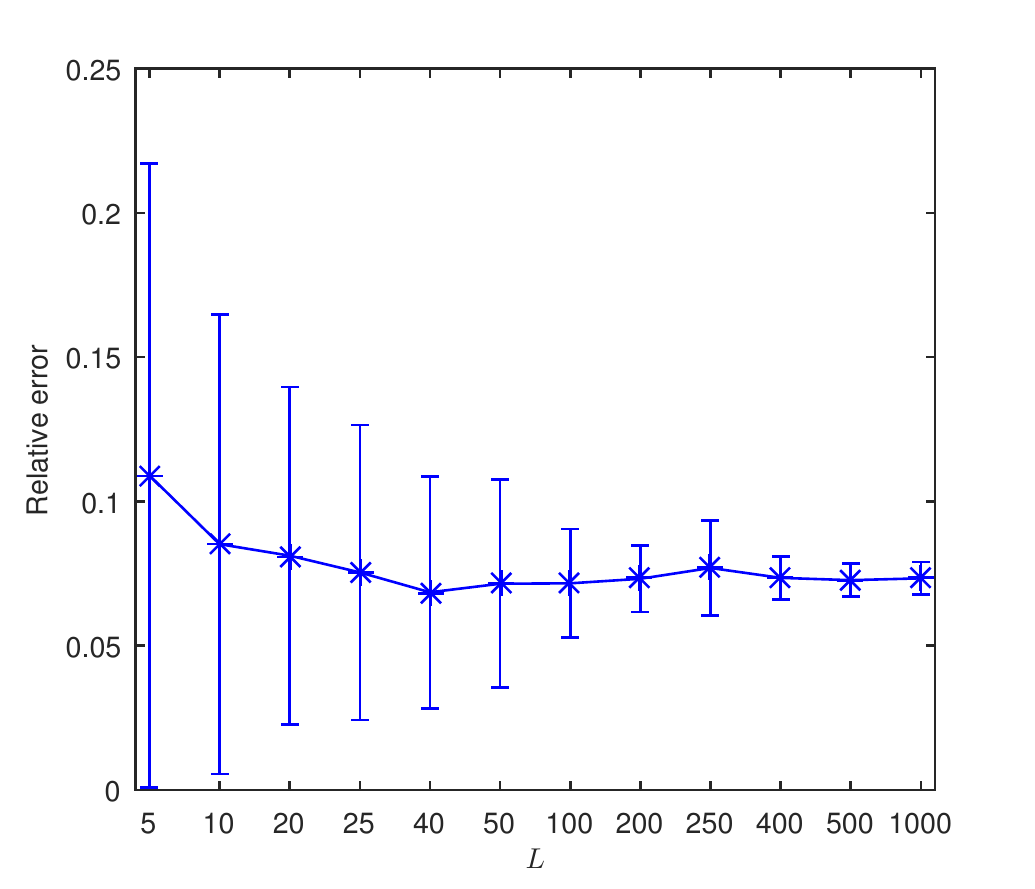}
	\caption{The relative errors of the two-stage stochastic homogenization method. (a)
		Test C with uniformly distributed coefficients;  (b) Test C with truncated normally distributed coefficients.
	}\label{fig:example3-2}
\end{figure}

\begin{table}[!htbp]
	\caption{\label{tab2}Total compute times for the three test cases by using
		the two-stage stochastic homogenization method (``Present") and the reference method (``Reference").}
	\footnotesize
	\begin{center}
		\begin{tabular}{ccccc}
			\toprule
			& Method & Total time(s) & Cell problem(s) & Homogenized problem(s)\\
			\midrule
			\multirow{2}{*}{Test A}   & Reference & 6,415.9 & 2,569.9 & 3,846.0\\
			& Present & 2,704.4 & 2,703.9 & 0.5\\
			\midrule
			\multirow{2}{*}{Test B}   & Reference & 10,115.5 & 6,258.3 & 3,857.2\\
			& Present & 6,337.3 & 6,334.4 & 2.9\\
			\midrule
			\multirow{2}{*}{Test C}   & Reference & 17,294.9  & 13,434.1 & 3,860.8\\
			& Present & 12,753.5 & 12,750.5 & 3.0\\
			\bottomrule
		\end{tabular}
	\end{center}
\end{table}

In each simulation, the unit cell $Q$ is partitioned into a body-fitted and  quasi-uniform mesh $\mathbb{T}^{\mathbf{k}}_h$. The mesh size $h=1/60$ is used for Type I of Test A  and $h=1/120$ for the other tests.
The equation \eqref{eq:new-2.29} is solved on a quasi-uniform partition
$\mathbb{T}_{h_0}$ of the domain  ${ D}$   with the mesh size $h_0=1/100$.
For comparison, we calculate the reference solution by solving \eqref{eq:new-2.29-1} on another quasi-uniform partition
$\mathbb{T}_{h_1}$ of the domain ${ D}$ with the mesh size $h_1=1/100$.
The relative error for the two-stage stochastic homogenization solution is defined as $\|u_0^{0,h_0}-  \hat u^{L,h_1}\|_{L^2( {D})}/\|  \hat u^{L,h_1}|_{L^2( {D})}$.
In Test A and Test C, the total number of samples  is taken as $10,000$.
Those samples are then divided into several subsets and they consist of respectively the following  numbers of samples $L=5,\ 10,\ 20,\ \cdots,\ 1,000$.
In Test B, the total number of samples is taken as $5,000$ and the numbers of
samples in the subsets are set as $L=5,\ 10,\ 20,\ \cdots,\ 500$.

The numerical expectations and fluctuations of the equivalent coefficients $\mu_{11,L}^0 $ for all three tests are presented in Figures
\ref{fig:example1-3}, \ref{fig:example2-1}, and \ref{fig:example3-1}. The numerical results, which are not shown, for the other entries of the equivalent matrix
are similar. We observe that as the number ($L$) of the total samples increases, the expectation tends to stabilize and the variance gradually decreases.
Moreover, we also observe that when $L=500\sim 1,000$ the two-stage stochastic homogenization method can obtain a stable and accurate equivalent matrix.
Figures \ref{fig:example1-4}, \ref{fig:example2-1}, and \ref{fig:example3-2}
show the relative errors of the computed two-stage stochastic homogenization solutions, and the numerical results show that the relative errors
reduce to a relatively low value (about 0.1\%-10\%) as the number of the total samples increases.
Figure \ref{fig:example1-5} displays the contour plots of the computed two-stage stochastic homogenization solutions and the reference solution. The consistency of
the two solutions is clearly seen.
Table \ref{tab1} presents the computational costs for solving the cell problem and the homogenized problem. Since the homogenized problem is only solved once in the two-stage stochastic homogenization method, the computational cost for the proposed method is much less than the other approach in which the homogenized problem must be solved $L$ times. The CPU times used by Test A, B, and C are respectively given in Table \ref{tab2}. Those numerical results demonstrate that the two-stage stochastic homogenization method is computationally quite efficient and accurate.

	\section{Conclusion}\label{sec-con}
	\label{sec:conclusion}
	In this paper, we developed a two-stage stochastic homogenization method for solving  diffusion equations with random fast oscillation coefficients. In the first stage, the
	proposed method constructs an equivalent matrix by solving a cell problem posed
	on the finite cell ${\cal Q}_M$. It was proved that the equivalent matrix converges
	to the stochastic homogenized matrix as the cell size goes to infinity.
	To balance the efficiency and accuracy, the proposed two-stage stochastic
	homogenization method usually chooses a suitable large cell and calculates the empirical mean by taking $L$ samples in the probability space.
	In the second stage, the approximation of the homogenized problem, which is a random
	diffusion problem, is solved by employing an efficient multi-modes Monte Carlo method,
	after having shown that the equivalent matrix can be rewritten as a small random  perturbation of some deterministic matrix. As a result, the proposed two-stage method
	provides an efficient procedure to obtain an approximation to the homogenized solution.  The efficiency and accuracy of the two-stage stochastic homogenization method were validated by several numerical experiments on some benchmark problems.

{\color{black}
In summary, we propose a new stochastic homogenization method with a two-stage procedure, which
is different with the classical stochastic homogenization method or “cut-off” procedure.
This appears to be the first attempt to separate the computational difficulty caused
by the spatial fast oscillation of the solution and that caused by the randomness of the
solution, so they can be overcome separately using different strategies. Besides, the convergence of the solution of the spatially homogenized equation (from the first
stage) to the solution of the original random diffusion equation is established and the
optimal rate of convergence is also obtained for the proposed multi-modes Monte Carlo
method. This appears to be the first attempt to obtain the explicit convergence order
for the stochastic homogenization method.
}

\renewcommand\theequation{A.\arabic{equation}}
\setcounter{equation}{0}
{
\section*{Appendix A: The property of random matrix $A(\frac x \varepsilon,\omega)$}
The random matrix $A(\frac x \varepsilon,\omega)$ of the  composite materials usually does not satisfy the small random perturbation or can not be rewritten into the desired form by the Karhunen-Lo\'{e}ve expansion. The precise statement is supported by the following example.
Thus, the MMC finite element method of \cite{feng2016multimodes} can not be applied directly to solving the random diffusion problem \eqref{eq:problem}.

\vspace{0.25cm}
\noindent\textbf{Example 1:}
Let $Q=Q_1\cup Q_2$ with $Q_1\cap Q_2=\emptyset$.  Assume $ {D}_1=\cup_{\mathbf{k}\in\mathbb{Z}^d} {D}\cap
\varepsilon(Q_1+\mathbf{k})$ and $ {D}_2=\cup_{\mathbf{k}\in\mathbb{Z}^d} {D}\cap \varepsilon(Q_2+\mathbf{k})$.
The entries of the random matrix $A(\frac x \varepsilon,\omega)$  are taken as
\begin{equation}\label{eq:hetercoeff11}
a_{ij} \Bigl(\frac x \varepsilon,\omega \Bigr) = \left\{ \begin{gathered}
a_{ij,1}(1+\omega )\quad x \in { {D} _1}, \hfill \\
a_{ij,2}(1+\omega )\quad x \in { {D}_2}. \hfill \\
\end{gathered}  \right.
\end{equation}
where $a_{ij,1}$ and $a_{ij,2}$ are two distinct constants and $\omega$ is a   uniformly distributed scalar random variable over $[0,1]$. The autocorrelation function of $a_{ij} (\frac x \varepsilon,\omega )$ is given by
\begin{equation}\label{eq:hetercoeff12}
\begin{gathered}
{\hbox{Cov}_a}(s,t) = \mathbb{E}\left[\left(a_{ij} \Bigl(\frac s \varepsilon,\omega \Bigr)- \mathbb{E}\left[a_{ij} (\frac s \varepsilon,\omega )\right]\right)
\left(a_{ij} \Bigl(\frac t \varepsilon,\omega \Bigr)- \mathbb{E}\left[a_{ij} \Bigl(\frac t \varepsilon,\omega \Bigr)\right]\right)\right], 
\\
\qquad\qquad= \left\{ \begin{gathered}
\frac 1 {12} a_{ij,1}a_{ij,2}\quad \quad \hbox{for}~s \in { {D}_1},~t \in { {D} _2},~\hbox{or}~s \in { {D}_2},~t \in { {D} _1},\hfill \\
\frac 1 {12} a_{ij,1}^2\quad \quad \hbox{for}~s \in {{D}_1},~t \in {{D} _1}, \hfill \\
\frac 1 {12}a_{ij,2}^2\quad \quad \hbox{for}~s \in {{D}_2},~t \in {{D} _2}. \hfill \\
\end{gathered}  \right.
\end{gathered}
\end{equation}
Therefore, ${\hbox{Cov}_a}(s,t)$ is a discontinuous function.
By Lemma 4.2 of  \cite{alexanderian2015brief}, the Karhunen-Lo\'{e}ve expansion does not give the required small random perturbation form for random matrix $A (\frac x \varepsilon,\omega )$.

Moreover, suppose that the $(i,j)$-component of $A (\frac x \varepsilon,\omega )$ can be
rewritten as
$a_{ij}(\frac x \varepsilon,\omega )= a_{ij}^0(x)+\delta a^1_{ij} (\frac x \varepsilon,\omega )$
with $a^1_{ij} \in L^2(\Omega,L^{\infty}({D}))$ satisfying
\begin{equation*}
\mathbb{P}\bigl\{\omega\in\Omega; \left\| a_{ij}^1(\omega)\right\|_{L^{\infty}({D})}\leq \bar{a}\bigr\}=1,
\end{equation*}
where $\bar a$ is a constant independent of small parameter $\delta$.
Thus, we have
$$
 a^1_{ij} \Bigl(\frac x \varepsilon,\omega \Bigr)=\frac {a_{ij}\bigl(\frac x \varepsilon,\omega \bigr)- a_{ij}^0(x)} \delta,
$$
Let $\omega_1\in \Omega_1=[0,\frac14]$ and $\omega_2\in\Omega_2=[\frac34,1]$, it follows that
\begin{align}
\Bigl\| a^1_{ij} \Bigl(\frac x \varepsilon,\omega _1\Bigr) \Bigr\|_{L^{\infty}({D})}
&+\Bigl\| a^1_{ij} \Bigl(\frac x \varepsilon,\omega _2 \Bigr) \Bigr\|_{L^{\infty}({D})},\\
&\geq \frac {\|{a_{ij}(\frac x \varepsilon,\omega_1 )- a_{ij}(\frac x \varepsilon,\omega_2 )\|_{L^{\infty}({D})}}} \delta,\\
&\geq \frac 1 {2\delta}\max\{a_{ij,1},\,a_{ij,2}\}.
\end{align}
which contradicts with the condition $\mathbb{P}\bigl\{\omega\in\Omega; \left\| a_{ij}^1(\omega)\right\|_{L^{\infty}({D})}\leq \bar{a}\bigr\}=1$
with $\bar a$ being independent of small parameter $\delta$.
Therefore,
the random matrix $A (\frac x \varepsilon,\omega )$ can not be rewritten as
$A_{ij}(\frac x \varepsilon,\omega )= A_{ij}^0(x)+\delta A^1_{ij} (\frac x \varepsilon,\omega )$.

}


\bibliographystyle{abbrv}

\begin{thebibliography}{10}

\bibitem{alexanderian2015brief}
{\sc A.~Alexanderian}, {\em A brief note on the karhunen-lo$\backslash$eve
  expansion}, arXiv preprint arXiv:1509.07526,  (2015).

\bibitem{anantharaman2012introduction}
{\sc A.~Anantharaman, R.~Costaouec, C.~L. Bris, F.~Legoll, and F.~Thomines},
  {\em Introduction to numerical stochastic homogenization and the related
  computational challenges: some recent developments}, in Multiscale modeling
  and analysis for materials simulation, World Scientific, 2012, pp.~197--272.

\bibitem{anantharaman2011numerical}
{\sc A.~Anantharaman and C.~Le~Bris}, {\em A numerical approach related to
  defect-type theories for some weakly random problems in homogenization},
  Multiscale Modeling \& Simulation, 9 (2011), pp.~513--544.

\bibitem{anantharaman2012elements}
{\sc A.~Anantharaman and C.~Le~Bris}, {\em Elements of mathematical foundations
  for numerical approaches for weakly random homogenization problems},
  Communications in Computational Physics, 11 (2012), pp.~1103--1143.

\bibitem{bensoussan2011asymptotic}
{\sc A.~Bensoussan, J.-L. Lions, and G.~Papanicolaou}, {\em Asymptotic analysis
  for periodic structures}, vol.~374, American Mathematical Soc., 2011.

\bibitem{blanc2007stochastic}
{\sc X.~Blanc, C.~Le~Bris, and P.-L. Lions}, {\em Stochastic homogenization and
  random lattices}, Journal de math{\'e}matiques pures et appliqu{\'e}es, 88
  (2007), pp.~34--63.

\bibitem{bourgeat2004approximations}
{\sc A.~Bourgeat and A.~Piatnitski}, {\em Approximations of effective
  coefficients in stochastic homogenization}, in Annales de l'IHP
  Probabilit{\'e}s et statistiques, vol.~40, 2004, pp.~153--165.

\bibitem{cioranescu1999introduction}
{\sc D.~Cioranescu and P.~Donato}, {\em An introduction to homogenization},
  vol.~17, Oxford University Press Oxford, 1999.

\bibitem{cui2005multi}
{\sc S.~Y. Cui, J.Z. and L.~Cao}, {\em Multi-scale analysis method for
  composite materials with random distribution and related structures}, 8-th
  Intern. Conf. Composites Engineering,  (2001), pp.~161--161.

\bibitem{duerinckx2016structure}
{\sc M.~Duerinckx, A.~Gloria, and F.~Otto}, {\em The structure of fluctuations
  in stochastic homogenization}, Commun. Math. Phys., 377 (2020), pp.~259--306.

\bibitem{feng2016multimodes}
{\sc X.~Feng, J.~Lin, and C.~Lorton}, {\em A multimodes monte carlo finite
  element method for elliptic partial differential equations with random
  coefficients}, International Journal for Uncertainty Quantification, 6
  (2016), pp.~429--443.

\bibitem{gilbarg2015elliptic}
{\sc D.~Gilbarg and N.~S. Trudinger}, {\em Elliptic partial differential
  equations of second order}, vol.~224, springer, 2015.

\bibitem{gitman2007representative}
{\sc I.~Gitman, H.~Askes, and L.~Sluys}, {\em Representative volume: Existence
  and size determination}, Engineering fracture mechanics, 74 (2007),
  pp.~2518--2534.

\bibitem{gloria2012optimal}
{\sc A.~Gloria and F.~Otto}, {\em An optimal error estimate in stochastic
  homogenization of discrete elliptic equations}, The annals of applied
  probability,  (2012), pp.~1--28.

\bibitem{graham2011quasi}
{\sc I.~G. Graham, F.~Y. Kuo, D.~Nuyens, R.~Scheichl, and I.~H. Sloan}, {\em
  Quasi-monte carlo methods for elliptic pdes with random coefficients and
  applications}, Journal of Computational Physics, 230 (2011), pp.~3668--3694.

\bibitem{hashin1962variational}
{\sc Z.~Hashin and S.~Shtrikman}, {\em A variational approach to the theory of
  the effective magnetic permeability of multiphase materials}, Journal of
  applied Physics, 33 (1962), pp.~3125--3131.

\bibitem{kozlov1979averaging}
{\sc S.~M. Kozlov}, {\em Averaging of random operators}, Matematicheskii
  Sbornik, 151 (1979), pp.~188--202.

\bibitem{li2005multi}
{\sc Y.~Li and J.~Cui}, {\em The multi-scale computational method for the
  mechanics parameters of the materials with random distribution of multi-scale
  grains}, Composites Science and Technology, 65 (2005), pp.~1447--1458.

\bibitem{2014Efficient}
{\sc J.~Liu, J.~Lu, and X.~Zhou}, {\em Efficient rare event simulation for
  failure problems in random media}, SIAM Journal on Scientific Computing, 37
  (2014).

\bibitem{oleinik2009mathematical}
{\sc O.~A. Ole{\"\i}nik, A.~Shamaev, and G.~Yosifian}, {\em Mathematical
  problems in elasticity and homogenization}, Elsevier, 2009.

\bibitem{papanicolaou1979boundary}
{\sc G.~C. Papanicolaou}, {\em Boundary value problems with rapidly oscillating
  random coefficients}, in Colloquia Math. Soc., Janos Bolyai, vol.~27, 1979,
  pp.~853--873.

\bibitem{pozhidaev1990error}
{\sc A.~Pozhidaev and V.~Yurinski{\u\i}}, {\em On the error of averaging
  symmetric elliptic systems}, Mathematics of the USSR-Izvestiya, 35 (1990),
  p.~183.

\bibitem{sab1992homogenization}
{\sc K.~Sab}, {\em On the homogenization and the simulation of random
  materials}, European journal of mechanics. A. Solids, 11 (1992),
  pp.~585--607.

\bibitem{suribhatla2011effective}
{\sc R.~Suribhatla, I.~Jankovic, A.~Fiori, A.~Zarlenga, and G.~Dagan}, {\em
  Effective conductivity of an anisotropic heterogeneous medium of random
  conductivity distribution}, Multiscale Modeling \& Simulation, 9 (2011),
  pp.~933--954.

\bibitem{tartar2009general}
{\sc L.~Tartar}, {\em The general theory of homogenization: a personalized
  introduction}, vol.~7, Springer Science \& Business Media, 2009.

\bibitem{to2020fft}
{\sc Q.-D. To and G.~Bonnet}, {\em Fft based numerical homogenization method
  for porous conductive materials}, Computer Methods in Applied Mechanics and
  Engineering, 368 (2020), p.~113160.

\bibitem{vel2010multiscale}
{\sc S.~S. Vel and A.~J. Goupee}, {\em Multiscale thermoelastic analysis of
  random heterogeneous materials: Part i: Microstructure characterization and
  homogenization of material properties}, Computational Materials Science, 48
  (2010), pp.~22--38.

\bibitem{wang1999mesoscopic}
{\sc Z.~Wang, A.~Kwan, and H.~Chan}, {\em Mesoscopic study of concrete i:
  generation of random aggregate structure and finite element mesh}, Computers
  \& structures, 70 (1999), pp.~533--544.

\bibitem{xu2009stochastic}
{\sc X.~F. Xu and X.~Chen}, {\em Stochastic homogenization of random elastic
  multi-phase composites and size quantification of representative volume
  element}, Mechanics of materials, 41 (2009), pp.~174--186.

\bibitem{wu2016Efficient}
{\sc Y.~N. Y.T.~Wu}, {\em An efficient approach to determine the effective
  properties of random heterogeneous materials}, High Temperature Ceramic
  Matrix Composites 8: Ceramic Transactions, 248 (2014), pp.~23--28.

\bibitem{yurinskii1986averaging}
{\sc V.~V. Yurinskii}, {\em Averaging of symmetric diffusion in random medium},
  Siberian Mathematical Journal, 27 (1986), pp.~603--613.

\end{thebibliography}


\begin{thebibliography}{}

\end{thebibliography}

\end{document}